\documentclass{amsart}

\usepackage{amsmath,amssymb,amsthm}
\usepackage{graphicx,tikz}

\newtheorem{thm}{Theorem}
\newtheorem*{proposition}{Proposition}

\newtheorem{lemma}{Lemma}

\theoremstyle{definition}

\theoremstyle{remark}

\begin{document}

\title[]{Towards Optimal Gradient Bounds for \\the Torsion Function in the plane}
\thanks{S.S. is supported by the NSF (DMS-1763179) and the Alfred P. Sloan Foundation.}

\author[]{Jeremy G. Hoskins}
\address{Department of Mathematics, Yale University, New Haven, CT 06511, USA}
\email{jeremy.hoskins@yale.edu}

\author[]{Stefan Steinerberger}
\address{Department of Mathematics, Yale University, New Haven, CT 06511, USA}
\email{stefan.steinerberger@yale.edu}

\begin{abstract}
Let $\Omega \subset \mathbb{R}^2$ be a bounded, convex domain and let $u$ be the solution of $-\Delta u = 1$ vanishing on the boundary $\partial \Omega$. The
estimate
$$ \| \nabla u\|_{L^{\infty}(\Omega)} \leq c |\Omega|^{1/2}$$
is classical. 
 We use the P-functional, the stability theory of the torsion function and Brownian motion to establish the estimate for a universal $c < (2\pi)^{-1/2}$. We also give a numerical construction showing that the optimal constant satisfies $c \geq 0.358$. The problem is important in different settings: (1) as the maximum shear stress in Saint Venant Elasticity Theory, (2) as an optimal control problem for the constrained maximization of the lifetime of Brownian motion started close to the boundary and (3) optimal Hermite-Hadamard inequalities for subharmonic functions on convex domains.
\end{abstract}

\maketitle

\section{Introduction and result}

We study a basic question for elliptic partial differential equations in the plane that arises naturally in a variety of contexts. Let $\Omega \subset \mathbb{R}^2$ be a convex domain and let
\begin{align*}
-\Delta u &= 1 \qquad \mbox{inside}~\Omega\\
u &= 0 \qquad \mbox{on}~\partial \Omega.
\end{align*}
The purpose of this paper is to study a shape optimization problem and to establish bounds on the gradient. More precisely, our main result is as follows.
\begin{thm}[Main Result] Let $\Omega \subset \mathbb{R}^2$ be a convex, bounded domain and let $u$ denote the solution of $-\Delta u = 1$ with Dirichlet boundary conditions. For some universal $c < 1/\sqrt{2\pi} \sim 0.398$,
$$ \| \nabla u \|_{L^{\infty}(\Omega)}  \leq c \cdot |\Omega|^{1/2}$$
and the constant $c$ cannot be replaced by 0.358.
\end{thm}
Though this is only a very slight improvement over the bound $ 1/\sqrt{2\pi}$ obtained in \cite{hermitemany}, our proof uses very different arguments, is somewhat robust 
and suggestive of future directions. 
One could presumably make the improvement explicit, though it would result in a very
small number (something like $\sim 10^{-10}$ or possibly even less as it depends on other constants for which only non-optimal estimates are available).

Small improvements of this nature are not unusual, we refer to Bonk \cite{bonk} or Bourgain \cite{bourgain} for other examples. As in those cases, the main advance
is not the new constant but the new ideas that allow for an improvement. We also have high-precision numerical results that imply an improved lower bound and possibly point the way to a better understanding
of the optimal shape.

\begin{figure}[h!]
\begin{tikzpicture}
\node at (0,0) {\includegraphics[width = 0.6\textwidth]{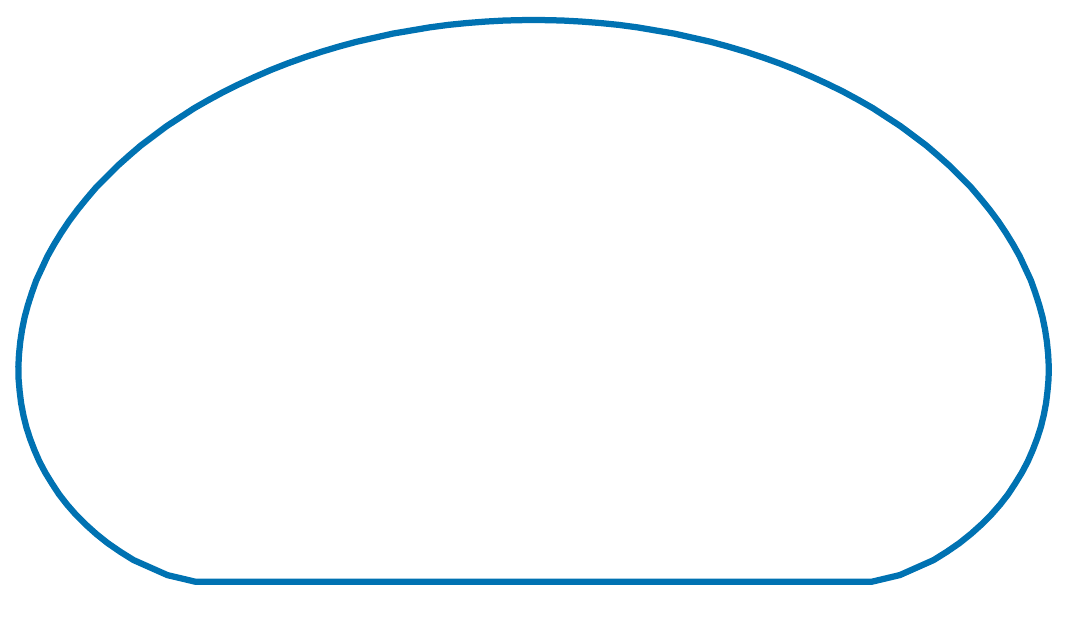}};
\end{tikzpicture}
\caption{The best example that we obtained through an optimization scheme sitting on top of high precision numerics (see \S 4): a domain $\Omega$ for which $ \| \nabla u\|_{L^{\infty}(\Omega)} \sim 0.358126008 | \Omega|^{1/2}$.} 
\end{figure}

We note that, in contrast to many other shape optimization problems, the optimal shape for this functional is quite distinct from a disk. We have a numerical example of a domain
that we believe to be close to optimal and we consider a more precise understanding of extremal domains to be an interesting problem. Is their boundary smooth? Is it true that their
curvature vanishes in exactly one point? 
It is noteworthy that this shape optimization problem appears (somewhat hidden) in a variety of different settings which we now survey. We 
believe it is one of the most fundamental shape optimization problems that is far from being understood.

\subsection{Saint Venant theory.} The function $u$ is also known as the \textit{torsion function} in the Saint Venant theory of elasticity \cite{st} from 1856. In particular, the maximum shear stress is known to occur
on the boundary (a 1930 result of Polya \cite{pol1}) 
$$ \tau = \left\| \frac{\partial u}{\partial n} \right\|_{L^{\infty}(\partial \Omega)} = \| \nabla u \|_{L^{\infty}(\Omega)}.$$
This quantity has been of substantial interest in elasticity theory, see e.g. \cite{bandle, pay0, pay1, payphi, payphi2, paywhe, phi, pol1, polya0, polind, sperb}. A classical inequality is \cite[Eq. 6.12]{sperb}
$$ \tau^2 \leq 2 \|u\|_{L^{\infty}(\Omega)}$$
and many different bounds on the maximum of the function are known (some of these bounds are discussed in \S 2.1). The best bound for arbitrary convex sets in terms of the volume
comes from \cite{hermitemany}
$$ \tau \leq \frac{|\Omega|^{1/2}}{\sqrt{2\pi}}.$$
An earlier result can be found in \cite{jianfeng}. Many more results are available if one is allowed to control more properties of $\Omega$, in particular if one has bounds on the curvature of $\partial \Omega$, see e.g. \cite{pay0, pay1, payphi,payphi2,paywhe,sperb}. As for lower bounds, ellipses are a particularly simple example because the torsion function can be computed in closed form (this was already pointed out in \cite{hermitemany, jianfeng}). A simple computation shows that for any $0 < a < 1$, the function
$$ u(x,y) = 1 - \frac{ax^2 + (1-a)y^2}{2}$$
is the torsion function on the set $\left\{(x,y) \in \mathbb{R}^2: u(x,y) \geq 0\right\}$ which is an ellipse with 
$$ |\Omega| = \frac{2 \pi}{\sqrt{a(1-a)}}.$$
Evaluating $\nabla u$ at $x= \sqrt{2/a}$ shows that
$$ \tau = \| \nabla u\|_{L^{\infty}(\Omega)} \geq \sqrt{2a}$$
and therefore the best constant $c$ in the estimate
$$ \tau = \| \nabla u\|_{L^{\infty}(\Omega)} \leq c |\Omega|^{1/2}$$
has to satisfy
$$ c \geq \max_{0 < a < 1} \frac{\sqrt{2a}}{ \sqrt{ \frac{2\pi}{\sqrt{a(1-a)}}}} = \frac{3^{3/4}}{4 \sqrt{\pi}} \sim 0.321\dots$$
This narrows down the optimal constant to lie in the range $c \in (0.321, 0.398)$. A priori, this looks like a small range already -- we emphasize
that it is in the nature of the problem for the largest derivative to be relatively stable under perturbations of the domain. Indeed, there is still a large variety of different convex bodies for which the largest derivative lies in that range. Our actual understanding of
the problem is still modest.

\subsection{Large Gradients: Historical Remarks.} We are interested in how \textit{big} the largest gradient can be. Saint Venant himself tried to understand the location of the largest gradient
since this is of interest in applications of elasticity theory: where is the maximal stress? Saint Venant 
writes 
\begin{quote}
Les points dangereux sont donc, comme dans l'ellipse et le rectangle, les
points du contour les plus rapproch\'es de l'axe de torsion, ou les extr\'emit\'es
des petits diam\`etre. (Saint Venant, \cite{st})
\end{quote}
More generally, on an ellipse the maximal derivatives are assumed on the short axis, a result considered `startling to many'  according to Thomson \& Tait in 1867 \cite{tt} (Thomson
would later be known as Lord Kelvin) who write in their \textit{Treatise on Natural Philosophy}
\begin{quote}
M. de St. Venant also calls attention to a conclusion from his solutions which to many may be startling, that in the simpler cases the places of greatest distortion are those points of the boundary
which are nearest to the axis [...] and the places of least distortion those farthest from it. (Thomson \& Tait \cite[\S 710]{tt})
\end{quote}
Boussinesq \cite{bous} gave a heuristic explanation in 1871. In 1900, Filon \cite{filon} confirms that the `fail points' (les points dangereux) for ellipses are along the shorter axis.
Around 1920, Griffith \& Sir G. I. Taylor \cite{griff} report on an appartus using soap bubbles to compute torsion.
The fact that the points with the largest gradient lie on the boundary was rigorously proven only in 1930 by Polya \cite{pol1}.

Saint Venant conjectured
that the maximum of the gradient in a convex domain is assumed in the point on the boundary where the largest inscribed circle intersects the boundary. Sweers \cite{stv2} showed
that the conjecture fails for either the set 
$$ \Omega = \left\{ (x_1, x_2) \in \mathbb{R}^2: (|x_1|+1)^2 + x_2^2 < 4, |x_2| \leq 1\right\}$$
or for one of the superlevel sets $\left\{x \in \Omega: u(x) \geq \varepsilon\right\}$ and thus fails in general (see also \cite{stv1, stv2, stv3, stv4}). However, the statement is known to hold under some additional assumptions (see e.g. Kawohl \cite{kawohl, kawohl2}).

\begin{center}
\begin{figure}[h!]
\begin{tikzpicture}[scale=1.8]
\draw [dashed](0,0) circle (1cm);
\draw [very thick] (0.73, -1) -- (-0.73, -1);
\draw [very thick] (0.73, 1) -- (-0.73, 1);
\draw [very thick] (-0.73, -1) to[out=117, in=243] (-0.73, 1);
\draw [very thick] (0.73, -1) to[out=63, in=297] (0.73, 1);
\end{tikzpicture}
\caption{The `barrel' domain of Sweers \cite{stv2}.}
\end{figure}
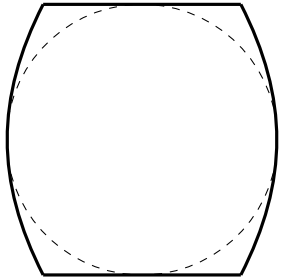
\end{center}

\subsection{A Benchmark PDE} The equation $-\Delta u = 1$ with Dirichlet boundary conditions inside convex domains has for a long time been a `benchmark PDE' for which
new types of techniques and results are being developed: in a certain sense, it is the simplest elliptic partial differential equation after $\Delta u = 0$. An example is the seminal result
of Makar-Limanov \cite{makar} who showed that solutions of $-\Delta u = 1$ in convex domains have the property that $\sqrt{u}$ is concave. Many more results have since
been established, we refer to the survey of Keady \& McNabb \cite{keady}. 
Results with a similar degree of precision are really only available for the first eigenfunction of the Laplacian. A fundamental result
by Brascamp \& Lieb \cite{brascamp} for the ground state of the Laplacian $-\Delta u = \lambda_1 u$ inside convex domains in the plane is convexity of the level set. Problems of this nature are of continued interest \cite{arango, finn, henriot, kawohl0, ma, payshape, sperb}.
We also emphasize that there are very precise results obtained for the structure of the first Laplacian eigenfunction by Grieser and Jerison \cite{grieser, grieser2, jerison} with
exciting recent developments by T. Beck \cite{beck1, beck2, beck3}. The usual question is: how much of the insight gained for these types of special solutions can be carried
over to more general solutions? 
We also mention the larger field of Shape Optimization (representative textbooks being Baernstein \cite{baern}, Lieb \& Loss \cite{lieb}, Henrot \cite{henriot}, Polya \& Szeg\H{o} \cite{polya})
concerned with the interplay of solutions of partial differential equations (or functionals of the solutions) with the geometry of the underlying domain. Both $-\Delta u = \lambda_1 u$
and $-\Delta u = 1$ have been actively investigated from that perspective and we see our contribution firmly aligned with this line of inquiry. In particular, our approach is sufficiently
robust to yield nontrivial results for larger families of PDE's. However, since we have not fully understood $-\Delta u = 1$, we have not pursued this further at this time.

\subsection{Expected Lifetime of Brownian motion.} The solution of $-\Delta u = 1$ with Dirichlet boundary conditions describes (up to constants depending on normalization) the expected
lifetime of Brownian motion. More presicely, $u(x)$ gives the expected lifetime of Brownian motion started in $x$ before hitting the boundary $\partial \Omega$. We refer to \cite{ban, ban2, ban3, ban4}
and references therein. There are many open problems. It is known, for example, that for simply connected $\Omega \subset \mathbb{R}^2$
$$ \frac{\mbox{inrad}(\Omega)^2}{2}  \leq \mbox{maximum expected lifetime} \leq c\cdot \mbox{inrad}(\Omega)^2.$$
Here, the constant $c$ is known to be 1 in convex domains \cite{sperb}, the extreme case is that of an infinite strip. Such results are only possible due to the potential-theoretic rigidity of
two dimensions, we refer to \cite{georgiev, lieb0, lierl, manas} for substitute results in higher dimensions.
 A result of a similar flavor is the following: among all convex sets with fixed
volume $|\Omega|$, the disk maximizes both the average expected lifetime (a 1948 result of Polya \cite{polya0}) as well as the maximum expected lifetime (this follows essentially from a rearrangement principle of Talenti \cite{talenti}).
For this type of problem, we also refer to recent work of Hamel \& Russ \cite{hamel} (see also J. Lu and the second author \cite{jianfeng2}).

\begin{center}
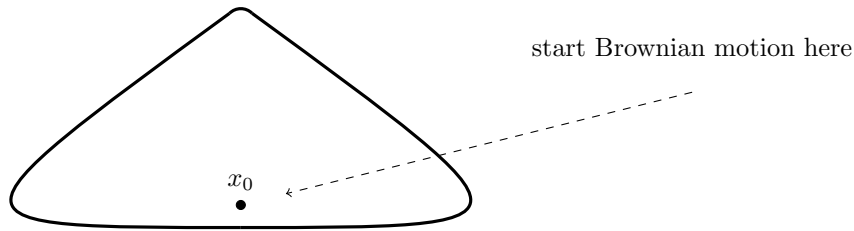
\begin{figure}[h!]
\begin{tikzpicture}[scale=3,rotate=90]
  \draw[very thick] (0,0) .. controls (0,1.35) .. (0.95,0.05);
    \draw[very thick] (0,0) .. controls (0,-1.35) .. (0.95,-0.05);
    \draw[very thick] (0.95, 0.05) to[out=315, in=45] (0.95, -0.05);
    \filldraw (0.1,0) circle (0.02cm);
    \node at (0.2,0) {$x_0$};
    \node at (0.8,-2) {start Brownian motion here};
    \draw [dashed,->] (0.6, -2) -- (0.15, -0.2);
\end{tikzpicture}
\caption{A domain $\Omega$ satisfying $\|\nabla u\|_{L^{\infty}(\Omega)} \sim 0.34|\Omega|^{1/2}$. Brownian motion started in the point $x_0$ close to the boundary has a (relatively) large expected lifetime compared to other domains.}
\end{figure}
\end{center}

The problem of sharp gradient estimates for the torsion function is equivalent to a natural and interesting problem in probability theory. As mentioned above, the shape maximizing the expected lifetime of Brownian motion among domains with fixed volume is
the disk (and, unsurprisingly, we start Brownian motion in the center of the disk). 
\begin{quote}
\textbf{Problem.} Among all convex domains of fixed volume $\Omega$, which one maximizes the expected lifetime of Brownian motion that starts in a point that is within $\varepsilon$ distance of the boundary?
\end{quote}
The Blaschke selection theorem guarantees the existence of these
extremal domains but it is not at all clear what sort of regularity properties they have. This problem depends on the size of $\varepsilon$, our paper is only concerned with the case $\varepsilon \rightarrow 0$. It is not too surprising that the expected lifetime is going to decay: starting close to the boundary makes it exceedingly likely to hit the boundary rather quickly. 
Nonetheless, there will always be Brownian paths that survive the initial dangerous phase and then explore the domain. Which shape guarantees the longest lifetime? It is intuitively clear that since
we start close to the boundary, that part of the boundary should be as flat as possible within the confines of being a convex domain of a fixed volume. This makes it seem natural to assume that the curvature of the boundary of the optimal domain vanishes in exactly one point. Our
numerical construction of a lower bound suggests certain shapes that might be close to extremal, we believe this to be an interesting open problem.

\begin{center}
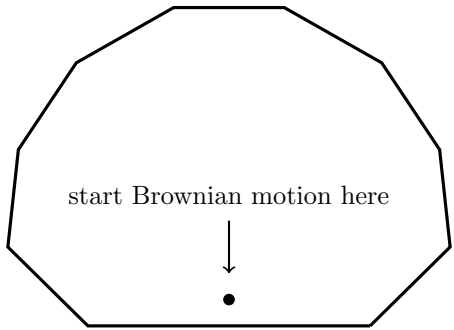
\begin{figure}[h!]
\begin{tikzpicture}[scale=3.5]
  \draw[very thick] (0.536, 0) -- (0.84, 0.3) -- (0.8, 0.67) -- (0.58, 1) -- (0.21, 1.21) -- (-0.21, 1.21) -- (-0.58, 1) -- (-0.8, 0.67) -- (-0.84, 0.3) -- (-0.536,0) -- (0.536,0);
  \filldraw (0,0.1) circle (0.02cm);
  \node at (0,0.5) {start Brownian motion here};
\draw [thick,->] (0,0.4) -- (0,0.2);
\end{tikzpicture}
\caption{A domain for which $\|\nabla u\|_{L^{\infty}(\Omega)} \sim 0.357|\Omega|^{1/2}$. Brownian motion started in the point $x_0$ close to the boundary has a (relatively) large expected lifetime compared to other domains.}
\end{figure}
\end{center}

We observe that these two examples (Fig. 3 and Fig. 4) have quite different shapes but have maximal gradients of a fairly similar size. As mentioned above, the problem comes with a great deal of stability which is echoed in these two numbers being rather similar. In particular, the difference of $1/50$ in these maximal gradients is actually, considering the underlying stability, a big difference. This further illustrates the difficulty of the problem: understanding the extreme domains for a quantity that is very stable.

\subsection{Sharp Hermite-Hadamard inequalities.} The Hermite-Hadamard inequality \cite{hada, hermite} is a (very) elementary fact: if $f:[a,b] \rightarrow \mathbb{R}$ is a convex function, then
$$ \frac{1}{b-a} \int_{a}^{b}{f(x) dx} \leq \frac{f(a) + f(b)}{2}.$$
There have been a very large number of variations on this inequality, we refer to \cite{bes, cal1, cal2, chen, square, dragomirk, nicu2, nicu,choquet, past} and the textbook \cite{drago}. Somewhat surprisingly, given
the large number of results and extensions, until recently there has been relatively little work on the general higher-dimensional case.
The second author \cite{stein1} proved that if $\Omega \subset \mathbb{R}^n$ is bounded and convex and $f:\Omega \rightarrow \mathbb{R}$ is convex and positive on $\partial \Omega$, then
$$ \frac{1}{|\Omega|} \int_{\Omega}{ f(x) dx} \leq \frac{c_n}{|\partial \Omega|} \int_{\partial \Omega}{ f(x) d\sigma},$$
where $c_n$ is a universal constant depending only on the dimension. This was improved to $n-1 \leq c_n \lesssim n^{3/2}$ in \cite{hermitemany}. Recently, Larson \cite{larson} established the sharp constant to be
$$ c_n = n.$$
The second author \cite{stein1} also showed that, under the same assumptions on $\Omega$ and $f$, we also have
$$  \int_{\Omega}{ f(x) dx} \leq c_n |\Omega|^{1/n} \int_{\partial \Omega}{ f(x) d\sigma}.$$
This was then improved by Jianfeng Lu and the second author \cite{jianfeng} who proved that $c_n \leq 1$ in all dimensions and that the inequality holds for the larger family of subharmonic functions $\Delta f \geq 0$ (every convex function is subharmonic). A slightly better estimate on the constant was given in \cite{hermitemany} which proved that in two dimensions $c_2 \leq 1/\sqrt{2\pi} \sim 0.39$. 
A characterization of the optimal constant as the largest gradient term of the torsion function is due to Niculescu \& Persson \cite{choquet}. The argument is as follows: since $f$ is subharmonic, i.e. $\Delta f \geq 0$, and $u \geq 0$ vanishes on the boundary, integration by parts shows
\begin{align*}
\int_{\Omega}{ f dx} = \int_{\Omega}{ f (-\Delta u) dx} &= \int_{\partial \Omega}{ \frac{\partial u}{\partial \nu} f d\sigma} - \int_{\Omega}{(\Delta f) u dx} \\
&\leq  \int_{\partial \Omega}{ \frac{\partial u}{ \partial \nu} f d\sigma} \\
&\leq \max_{x \in \partial \Omega}{  \frac{\partial u}{\partial \nu}(x)} \int_{\partial \Omega} { f d\sigma},
\end{align*}
where $\nu$ is the inward pointing normal vector and we used the condition $f \big|_{\partial \Omega} \geq 0$ in the last step. Moreover, by letting $f$ be
the harmonic extension of a characteristic function in a small part of the domain where $\partial u/\partial \nu$ is close to maximal, we see that this constant is optimal.
This characterization, combined with our result, implies the following improved Hermite-Hadamard inequality.

\begin{thm} Let $\Omega \subset \mathbb{R}^2$ be a convex domain and let $f: \Omega \rightarrow \mathbb{R}$ be a subharmonic function, $\Delta f \geq 0$, assuming positive values on $\partial \Omega$. Then, for some universal constant $c < 1/\sqrt{2\pi} \sim 0.398\dots$,
$$ \int_{\Omega}{ f(x) dx} \leq c \cdot |\Omega|^{1/2} \int_{\partial \Omega}{f(x) d\sigma}$$
and the constant $c$ cannot be replaced by $0.358$.
\end{thm}

We emphasize that this problem is completely equivalent to the other two problems: optimal gradient estimates for the torsion function and the largest lifetime of Brownian time constrained to starting close to the boundary. We do not have any knowledge about the extremal domain for this Hermite-Hadamard inequality. We do not know anything about the regularity of the boundary and we know nothing about the behavior of the curvature. However, courtesy of the underlying intuition coming from maximizing constrained Brownian motion, we have numerical examples of domains that we believe to be close to optimal.

\subsection{Our Approach.} We conclude by explaining our approach to prove Theorem 1. It is based on the notion of Fraenkel asymmetry: for any domain $\Omega \subset \mathbb{R}^n$, its Fraenkel asymmetry $\mathcal{A}(\Omega)$ is defined by
$$ \mathcal{A}(\Omega) = \inf_{|B| = |\Omega|}\frac{ \left| B \Delta \Omega \right|}{|\Omega|},$$
where $B$ ranges over all balls in $\mathbb{R}^n$ that have the same volume as $\Omega$ and 
$A \Delta B = (A \setminus B) \cup (B \setminus A)$ is the symmetric difference of two sets. The Fraenkel
asymmetry satisfies $0 \leq \mathcal{A}(\Omega) \leq 2$ and is a quantitative measure of how close a set
is to a ball.
Combining the P-functional of Sperb \cite{sperb}, an old inequality of Larry Payne \cite{paynegood} and a recent
quantitative Saint Venant theorem due to Brasco, De Philippis and Velichkov \cite{brasco}, we can establish that
for any convex $\Omega \subset \mathbb{R}^2$ with area $|\Omega|=1$, the solution of $-\Delta u = 1$ with Dirichlet
boundary conditions, satisfies, for some universal constant $\tau>0$, 
$$ \|\nabla u\|_{L^{\infty}(\Omega)} \leq \frac{1}{\sqrt{2\pi}} - \tau \cdot \mathcal{A}(\Omega)^2.$$
This etablishes a uniform improvement for all domains satisfying $\mathcal{A}(\Omega) \in (\varepsilon_0, 2)$ for any $\varepsilon_0 > 0$.
It remains to understand the case where $\Omega$ is very close to the disk in the sense of $\mathcal{A}(\Omega) \leq \varepsilon_0$.
Here we will establish the following stability result.

\begin{proposition} There exists $\varepsilon_0 > 0$ such that if $\Omega \subset \mathbb{R}^2$ is convex, $|\Omega| = 1$, $\mathcal{A}(\Omega) \leq \varepsilon_0$
and $u$ solves $-\Delta u = 1$ with Dirichlet boundary conditions set to 0, then
$$ \| \nabla u\|_{L^{\infty}(\Omega)} \leq 0.39 < \frac{1}{\sqrt{2\pi}}.$$
\end{proposition}
This raises several interesting open problems: is there a form of stability theory for the maximal gradient of the torsion function in convex domains?
 More precisely, one could ask whether for all convex $\Omega$
with $|\Omega|=1$, the solution of $-\Delta u = 1$ with Dirichlet boundary conditions satisfies
$$ \| \nabla u\|_{L^{\infty}(\Omega)} \leq \frac{1}{2\sqrt{\pi}} + c\cdot \mathcal{A}(\Omega)^{\alpha}$$
for some fixed constants $\alpha$ and $c$. One might expect this quantity to be quite stable. Indeed,
one of the reasons why understanding extremal domains for this problem might be difficult is precisely this form of stability
that is also reflected in our numerical examples.
 A more general question would be whether, for any two convex domains $\Omega_1, \Omega_2 \in \mathbb{R}^2$ having the same
measure, the solutions $u_1, u_2$ satisfy a general stability property along the lines of
$$ \left| \| \nabla u_1\|_{L^{\infty}(\Omega_1)} - \| \nabla u_2\|_{L^{\infty}(\Omega_2)} \right| \leq c\cdot \left( \frac{| (\Omega_1 \setminus \Omega_2) \cup (\Omega_2 \setminus \Omega_1)|}{|\Omega_1|}\right)^{\beta}.$$

\section{Observations and Remarks} We first summarize the existing arguments surrounding the P-functional approach, use them to prove the existence of an extremizing domain, discuss a seemingly novel isoperimetric inequality and then
proceed to establish the main result in \S 3.

\subsection{Summary of the existing argument.}

We will present the relevant material for the equation $-\Delta u = 2$ with Dirichlet boundary conditions inside a convex domain. Our object of
interest, gradient bounds for solutions of $-\Delta u = 1$, then follow from a rescaling by a factor of $2$. It is slightly more convenient to work with $-\Delta u = 2$
since that is the scaling corresponding to the expected lifetime of Brownian motion.
The P-functional associated with solutions of $-\Delta u =2$ is 
$$ P(u) = |\nabla u|^2 + 4u.$$
The important property is that $P(u)$ assumes its maximum at the unique global maximum of $u$. Moreover, it is known that $|\nabla u|$ assumes its largest value on the boundary (see \cite{pol1, sperb}). These two properties combined imply that
$$ \| \nabla u\|_{L^{\infty}(\Omega)}^2 \leq 4\|u\|^{}_{L^{\infty}(\Omega)}.$$
We repeat an argument from the book of Sperb \cite{sperb}. We have, for all $x \in \Omega$,
$$ |\nabla u(x)| \leq 2 \sqrt{\|u\|_{L^{\infty}(\Omega)} - u(x)}.$$
Moreover, the gradient is larger than any directional derivative. Integrating this identity along
the shortest line connecting the point $x_0$ at which $u$ assumes its maximum to the closest
point on the boundary shows that
$$ \|u\|_{L^{\infty}(\Omega)} \leq d(x_0, \partial \Omega)^2 \leq  \mbox{inrad}(\Omega)^2.$$
However, one does not necessarily need to integrate along a line leading to the boundary; integrating over lines
leading to an arbitrary point in the domain shows that
$$ \|u\|_{L^{\infty}(\Omega)}  \leq \|x-x_0\|^2 + u(x).$$
Integrating over the entire domain $\Omega$, which we assume to have volume 1, leads to
$$  \|u\|_{L^{\infty}(\Omega)}  \leq \int_{\Omega}{ \|x-x_0\|^2 dx} + \int_{\Omega}{u(x) dx}.$$
This is a known result and can be found in \cite[Eq. (6.14)]{sperb}. 
The quantity
$$ \int_{\Omega}{u(x) dx} \qquad \mbox{is also known as the torsional rigidity}$$
and is a well-studied object. Saint Venant \cite{st} conjectured in the 1850s that among all domains of fixed area, the torsional rigidity is
maximized by the circle (now known as Saint Venant's Theorem). The first rigorous proof seems to have been given in a 1948 paper of
Polya \cite{polya0}. Davenport gave another proof (in \cite{polya}), a third proof is due to Makai \cite{makai}.
Nowadays, it is often considered a consequence of Talenti's rearrangement principle \cite{talenti}. For the purpose
of obtaining an upper bound, it thus suffices to assume $\Omega$ is a disk of radius $\rho = 1/\sqrt{\pi}$ in which case
 $$ \mbox{the torsion function is explicit:} \qquad  \frac{1}{2\pi} - \frac{1}{2}\|x\|^2$$
and thus
$$\int_{\Omega}{u(x) dx} \leq \int_{0}^{1/\sqrt{\pi}}{ \left( \frac{1}{2\pi} - \frac{r^2}{2} \right) 2\pi r dr} = \frac{1}{4\pi}.$$
Moreover, as implied by Talenti's rearrangement principle \cite{talenti}, the maximum value of $u$ also increases under symmetrization and thus,
for all domains of area $|\Omega|=1$, we have
$$ \|u\|_{L^{\infty}(\Omega)} \leq \frac{1}{2\pi}.$$

Summarizing these ideas, we see that if $\Omega \subset \mathbb{R}^2$ is a convex set of unit area and $-\Delta u = 2$ in $\Omega$ with Dirichlet boundary conditions, then
\begin{align*}
\| \nabla u\|_{L^{\infty}(\Omega)} &\leq 2\|u\|^{1/2}_{L^{\infty}(\Omega)} \leq \frac{\sqrt{2}}{\sqrt{\pi}}\\
\| \nabla u\|_{L^{\infty}(\Omega)} &\leq 2~\mbox{inrad}(\Omega)\\
\| \nabla u\|_{L^{\infty}(\Omega)}^2 &\leq \frac{1}{\pi} + 4\int_{\Omega}{ \|x-x_0\|^2 dx}.
\end{align*}
The first bound yields, after scaling by a factor of 2, the upper bound of $1/\sqrt{2\pi} \sim 0.39\dots$ that we wish to improve upon.
The second inequality would imply an improvement as soon as $\mbox{inrad}(\Omega) < 1/\sqrt{2\pi}$ (however, $|\Omega|=1$ only yields
$\mbox{inrad}(\Omega) < 1/\sqrt{\pi}$). The third
estimate, coupled with an isoperimetric estimate on the polar momentum, does not yield any improvement either.

 \subsection{An Isoperimetric Principle for Polar Momentum} 
 In the purpose of investigating various ways of utilizing $P-$functional arguments, we did come across an amusing isoperimetric principle for the polar momentum that
 we could not find in the literature.
\begin{lemma} Let $0 < \rho \leq 1/\sqrt{\pi}$. Among all convex sets $\Omega \subset \mathbb{R}^2$ with area 1 that contain a disk of radius $\rho$ around the origin, the functional
$$ J(\Omega) = \int_{\Omega}{ \|x\|^2 dx}$$
is maximized if $\Omega$ is the convex hull of a disk of radius $\rho$ and a point (the point is uniquely defined by the volume constraint $|\Omega| = 1$ up to rotation invariance).
\end{lemma}

There is nothing particularly special about the weight $\|x\|^2$ and the statement holds if $\|x\|^2$ is replaced by a strictly monotonically increasing function . 

\begin{center}
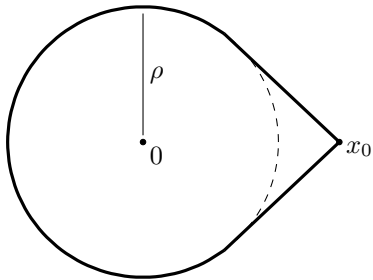
\begin{figure}[h!]
\begin{tikzpicture}[scale=0.9]
\filldraw (2.9,0) circle (0.04cm);
\draw [very thick] (2.9,0) -- (1.2, 1.6);
\draw [very thick] (2.9,0) -- (1.2, -1.6);
\node at (3.2,-0.1) {$x_0$};
\filldraw (0,0) circle (0.04cm);
\node at (0.2, -0.2) {0};
\draw (0, 0.1) -- (0, 1.9);
\node at (0.2, 1) {$\rho$};
\draw [very thick,domain=53:307, smooth] plot ({2*cos(\x)}, {2*sin(\x)});
\draw [dashed, domain=0:53, smooth] plot ({2*cos(\x)}, {2*sin(\x)});
\draw [dashed, domain=307:360, smooth] plot ({2*cos(\x)}, {2*sin(\x)});
\end{tikzpicture}
\caption{The extremal configuration: the distance of $x_0$ is uniquely determined by $\rho$ and $|\Omega| = 1$.}
\end{figure}
\end{center}

The Blaschke selection theorem immediately implies the existence of a (not necessarily unique) extremizing convex domain for this problem which we call  $\Omega$. It remains to understand its properties. We introduce $f:(0, \infty) \rightarrow \mathbb{R}$ via
$$ f(r) = \left| \left\{x \in \Omega: \|x\| = r\right\} \right|.$$
We have
$$ \int_{0}^{\infty}{f(r) dr} = 1 \qquad \mbox{and want to maximize} \qquad \int_{0}^{\infty}{f(r) r^2 dr}.$$
We know that there exists a constant $M(\rho)$ such that $f(r) = 0$ for all $r > M(\rho)$ (since $\Omega$ is a convex set contains a disk of radius $\rho$, it cannot have an arbitrarily large diameter since that would violate $|\Omega|=1$).
Integration by parts yields
\begin{align*}
 \int_{0}^{M(\rho)}{f(r) r^2 dr} &= F(r) r^2 \big|_{0}^{M(\rho)} - 2 \int_{0}^{M(\rho)}{ F(r) r dr}\\
 &= M(\rho)^2 - 2 \int_{0}^{M(\rho)}{ F(r) r dr}.
\end{align*}
This shows that instead of maximizing the integral over $f(r) r^2$, we could instead try to minimize the integral over $F(r) r$, where $F(r)$ is the amount of area of $\Omega$ of distance at most $r$ from the origin. This, however, turns out to have a relatively simple solution for each value of $r$ that happens to not depend on $r$. In particular, Lemma 1 is implied by a (stronger) geometric statement which we now prove.

\begin{lemma} Let $0 < \rho \leq 1/\sqrt{\pi}$. For each $r > \rho$, among all convex sets $\Omega \subset \mathbb{R}^2$ with area 1 that contain a disk of radius $\rho$ around the origin, the area
$$ \left| \left\{x \in \Omega: \|x\| \leq r\right\}\right| \qquad \mbox{is minimized}$$
 if $\Omega$ is the convex hull of a disk of radius $\rho$ and a point (the point is uniquely defined by the volume constraint $|\Omega| = 1$ up to rotational invariance).
\end{lemma}
\begin{proof}
There is a total of 
$$  0 < |\Omega| - \rho^2 \pi = 1 - \rho^2 \pi$$
area outside the disk of radius $\rho$ and we assume that 
$$ A =  \left| \left\{x \in \Omega: \rho \leq \|x\| \leq r\right\}\right| \qquad \mbox{and} \qquad B =  \left| \left\{x \in \Omega:  \|x\| \geq r\right\}\right|.$$
We are interested in understanding which convex domain minimizes $A$ or, since $A+B = 1 - \rho^2 \pi$ is fixed, which shape maximizes $B$. 
We introduce a third quantity, the length $\ell$ of the level set $r$,
$$ \ell = \left| \left\{ x \in \Omega: \|x\| =r \right\} \right|.$$

\begin{center}
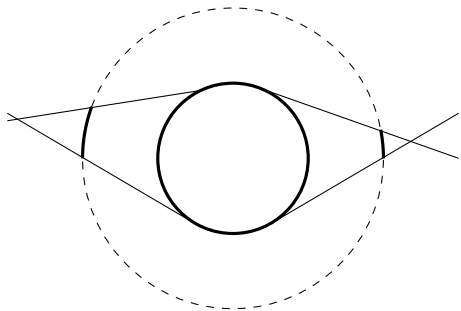
\begin{figure}[h!]
\begin{tikzpicture}
\draw [very thick] (0,0) circle (1cm);
\draw [dashed] (0,0) circle (2cm);
\draw [very thick, domain=0:11, smooth] plot ({2*cos(\x)}, {2*sin(\x)});
\draw [very thick, domain=160:180, smooth] plot ({2*cos(\x)}, {2*sin(\x)});
\draw (0.45, -0.9) -- (3,0.6);
\draw (0.45, 0.9) -- (3,0);
\draw (-0.45, -0.9) -- (-3,0.6);
\draw (-0.45, 0.9) -- (-3,0.5);
\end{tikzpicture}
\caption{The circle of radius $\rho$, the set $\left\{x \in \Omega: \|x\|= r\right\}$ and the induced cones.}
\end{figure}
\end{center}

We first argue that, depending on the size of $B$, $\ell$ cannot be too small. This can be done as follows: a priori we have no clear idea about the
structure of $\left\{ x \in \Omega: \|x\| =r \right\}$. For example, it could be comprised
of several circular arcs. We note that it could a priori also be comprised of an entire circle, however, we are interested in lower bounds on $\ell$ in terms of $B$ so this case is of less interest.
Suppose now that  $\left\{ x \in \Omega: \|x\| =r \right\}$ is comprised of several circular arcs. Then the fact that the endpoints of the circular arcs are endpoints implies, together with convexity, the existence of cone structures in which the set $\left\{ x \in \Omega: \|x\| > r \right\}$ must be contained, thus giving an upper bound on $B$. Moreover, we note that this upper bound is maximal if $\left\{ x \in \Omega: \|x\| =r \right\}$ is comprised of a single circular arc of length $\ell$. This implies an upper bound on $B$ in terms of $\ell$ or, conversely, it implies a lower bound on $\ell$ given $B$. Moreover, the convex hull of the disk and a point shows that this bound is sharp. We introduce $f(x)$ as the smallest possible length that $\ell$ can have if $|B| =x$, i.e.
$$ f(x) = \inf \left\{ \ell: |B| = x\right\}.$$
We have just shown that the extremal case occurs when $B$ has a cone structure. This would allow us to explicitly compute $f$ as a geometric quantity but we will only need a very basic property: $f$ is monotonically increasing. \\
As for the second part of the argument, we observe that there must be a lower bound on $A$ given $\ell$ (see the same Figure). Each element $x\in \Omega$ with $\|x\|=r$ induces, by convexity, a region which we know has to be a subset of $A$: this is the convex hull between the point $x$ and the disk. The total amount of area induced is minimized if these induced areas overlap as much as possible which happens when $\left\{ x \in \Omega: \|x\| =r \right\}$ is a circular arc. From this, we can define
$$ g(x) = \inf \left\{ |A|: \ell = x \right\}.$$
Again, by the previous argument we can actually determine $g$ in terms of trigonometric functions. However, once more the precise form of $g$ is not important, only its monotonicity is. The crucial ingredient is now that for both $f$ and $g$ the extremal configuration is actually the same is: both are sharp for the convex hull of a disk and a point outside. However, $A+B  = 1-\rho^2 \pi$. We observe that, by definition,
$$ \ell \geq f(B).$$
Then, by monotonicity of $g$, we have
$$ A \leq g(f(B)).$$
Since $A+B =  1-\rho^2 \pi$, we have
$$ B \leq 1 - \rho^2 \pi - g(f(B)).$$
This provides an upper bound on how big $B$ can be. Since all these inequalities are sharp for the case of $\Omega$ being the convex hull of a disk and a point outside, we obtain the desired result.
\end{proof}

\subsection{A Remark on the $P-$functional.}
Every bound in this paper is a consequence of Sperb's $P-$functional (for solutions of $-\Delta u = 1$)
$$ P(u) = |\nabla u|^2 + 2 u$$
and the fact that $P(u)$ assumes its maximum in a critical point (which is unique and the point where $u$ is maximal). Since $|\nabla u|$ assumes
its maximum on the boundary, this implies 
$$ \left\| \nabla u \right\|_{L^{\infty}(\Omega)}^2 \leq 2 \left\| u\right\|_{L^{\infty}(\Omega)}.$$
We illustrate this numerically for a domain $\Omega$ on which we obtain
$$ \| \nabla u\|_{L^{\infty}(\Omega)} \sim 0.358 | \Omega|^{1/2}$$
and which we believe to be somewhat close to the extremal case. The inequality is, naturally valid, however, it is not close to optimal since
 $$0.128 \sim  \left\| \nabla u \right\|_{L^{\infty}(\Omega)}^2 \leq 2 \left\| u\right\|_{L^{\infty}(\Omega)} \sim 0.15.$$

\begin{figure}[h!]
\begin{minipage}[l]{.48\textwidth}
\begin{tikzpicture}
\node at (0,0) {\includegraphics[width = 1.1\textwidth]{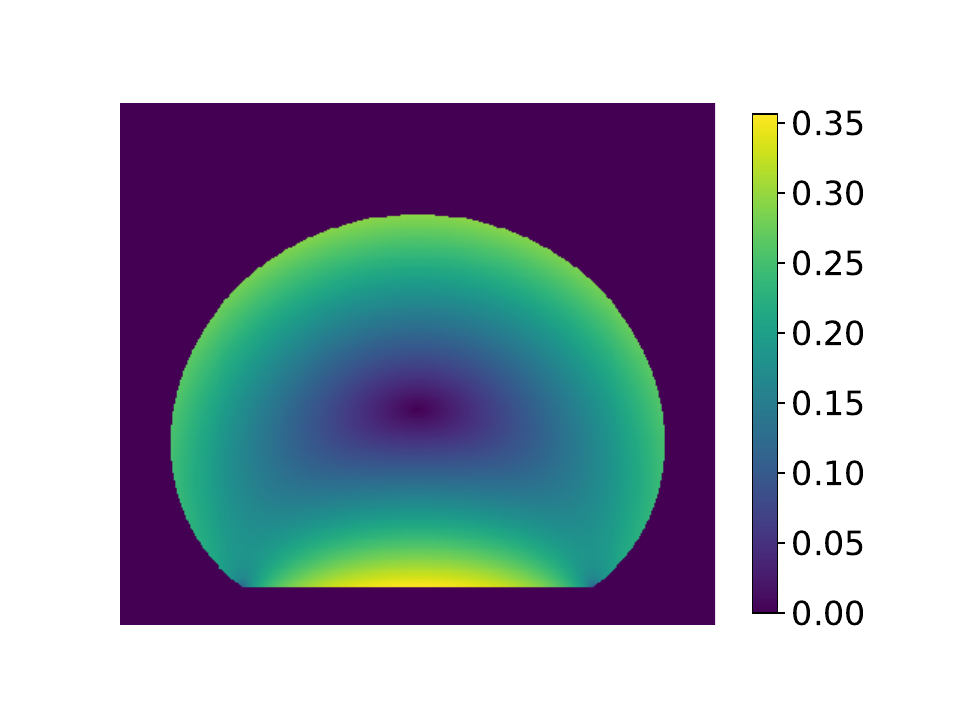}};
\end{tikzpicture}
\end{minipage} 
\begin{minipage}[r]{.48\textwidth}
\begin{tikzpicture}
\node at (0,0) {\includegraphics[width = 1.1\textwidth]{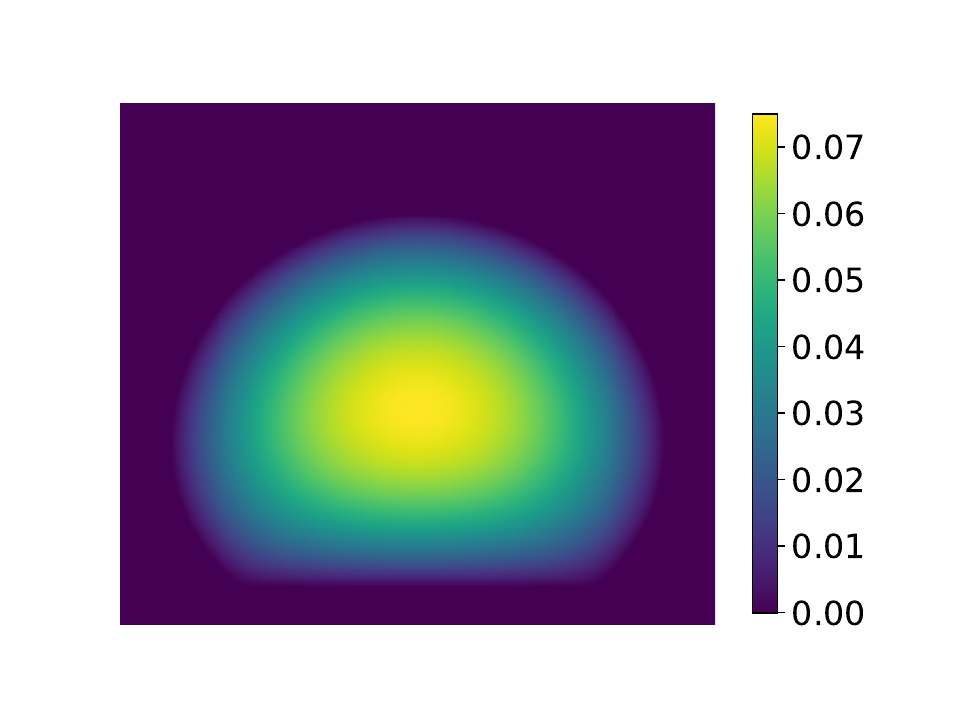}};
\end{tikzpicture}
\end{minipage} 
\caption{A domain $\Omega$ for which $ \| \nabla u\|_{L^{\infty}(\Omega)} \sim 0.358 | \Omega|^{1/2}$ and the size of derivative (left) and the size of the solution.} 
\end{figure}

This leads to a natural observation: the argument on which we base all our estimates of the normal derivative is intrinsically tied to the quality of the $P-$functional: in particular, if the $P-$functional is not close to a constant, then our argument becomes lossy in this step and we will be unable to obtain sharp estimates using this approach.

\begin{figure}[h!]
\begin{minipage}[l]{.48\textwidth}
\begin{tikzpicture}
\node at (0,0) {\includegraphics[width = 1.1\textwidth]{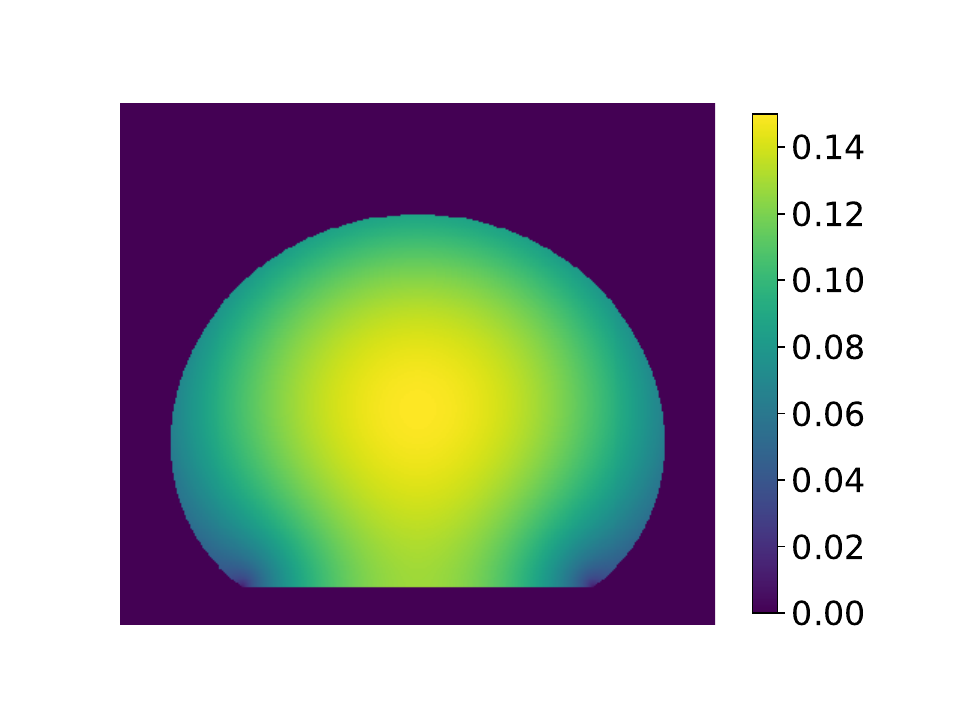}};
\end{tikzpicture}
\end{minipage} 
\begin{minipage}[r]{.48\textwidth}
\begin{tikzpicture}
\node at (0,0) {\includegraphics[width = 1.1\textwidth]{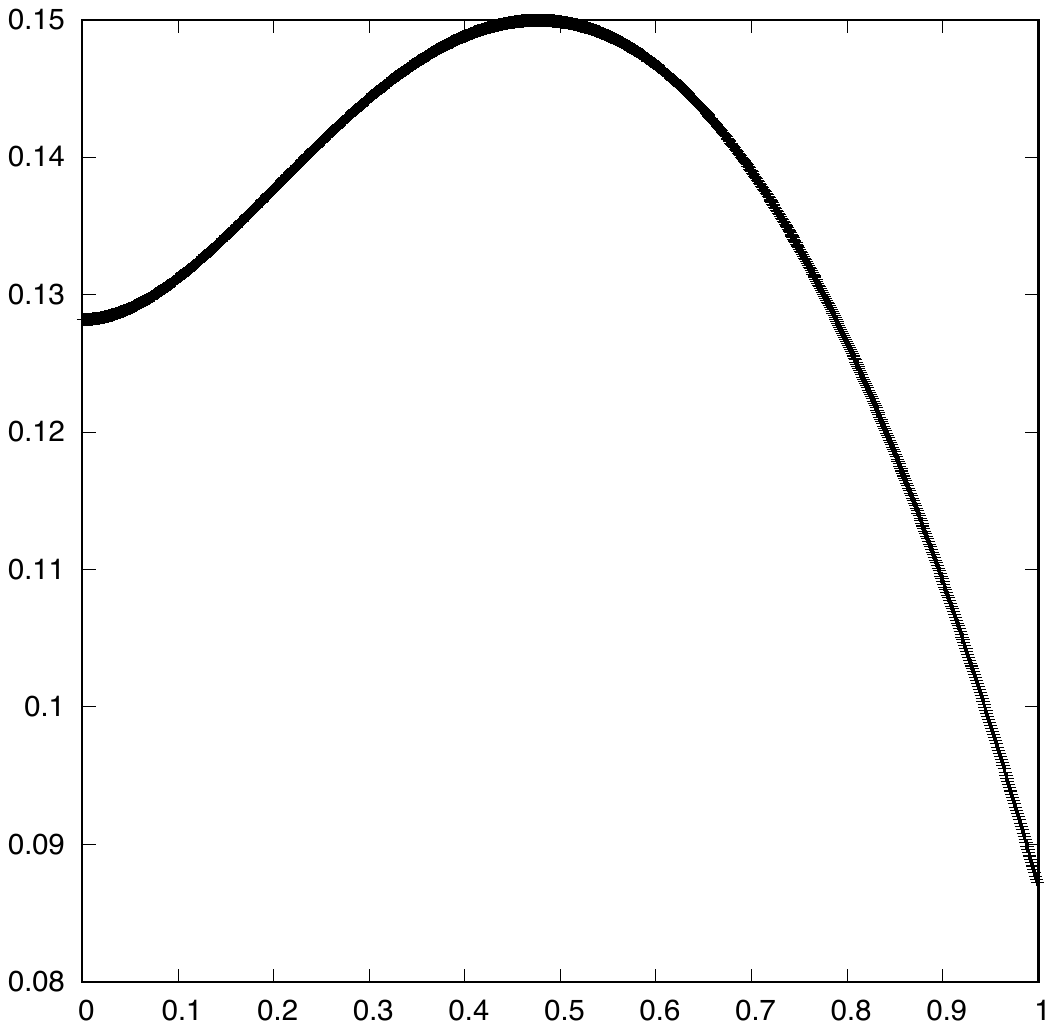}};
\end{tikzpicture}
\end{minipage} 
\caption{A domain $\Omega$ for which $ \| \nabla u\|_{L^{\infty}(\Omega)} \sim 0.358 | \Omega|^{1/2}$ and the size of $P-$functional on the domain (left) as well as restricted to the line segment that starts in the point of maximal derivative and goes through the global maximum (right).} 
\end{figure}

A natural question is thus: what sort of arguments are at our disposal when it comes to obtaining sharp bounds on the derivative of $-\Delta u = 1$? Can the $P-$functional approach be refined? Are there other arguments? We note that the argument of Jianfeng Lu and the second author \cite{jianfeng} does not require the $P-$functional but faces other challenges. We believe this to be an interesting problem.

\section{Proof of Theorem 1 and Theorem 2}

\subsection{Outline of the Argument.} The argument is in two parts. To state it, we recall the notion of Fraenkel asymmetry
$\mathcal{A}$ of a set $\Omega \subset \mathbb{R}^n$: in short, the Fraenkel asymmetry is defined by
$$ \mathcal{A}(\Omega) = \inf_{|B| = |\Omega|}\frac{ \left| B \Delta \Omega \right|}{|\Omega|},$$
where $B$ ranges over all balls in $\mathbb{R}^n$ that have the same volume as $\Omega$ and 
$A \Delta B = (A \setminus B) \cup (B \setminus A)$ is the symmetric difference of two sets. The Fraenkel
asymmetry satisfies $0 \leq \mathcal{A}(\Omega) \leq 2$ and is a quantitative measure of how close a set
is to a ball.
\begin{enumerate}
\item We show that if $\mathcal{A}(\Omega) \in (\varepsilon, 2)$, then
$$ \| \nabla u\|_{L^{\infty}} \leq \frac{1}{\sqrt{2\pi}} - c_{\varepsilon},$$
where $c_{\varepsilon} > 0$ is a constant depending only on $\varepsilon$.
\item This first part implies that it suffices to prove 
$$ \| \nabla u\|_{L^{\infty}} \leq \frac{1}{\sqrt{2\pi}} - c$$
for some universal $c>0$ under the additional assumption $\mathcal{A}(\Omega) \in (0,\varepsilon)$ where
$\varepsilon$ is a parameter that we can choose (though it has to be positive and universal). This we do via a probabilistic argument.
\end{enumerate}

\subsection{Part 1: Fraenkel asymmetry is large.}
The first part of the argument, an improved gradient estimate for domains that have a large Fraenkel asymmetry, is easy to describe. 
\begin{proposition} There exists a universal constant $c>0$ such that if $\Omega \subset \mathbb{R}^2$ is convex with $|\Omega| = 1$ and $-\Delta u = 1$ with Dirichlet boundary conditions, then
$$ \| \nabla u\|_{L^{\infty}(\Omega)} \leq \frac{1}{\sqrt{2\pi}} - c \cdot \mathcal{A}(\Omega)^3.$$
\end{proposition}
\begin{proof}

We will, throughout the rest of the argument, work with $-\Delta u = 2$ to simplify exposition (and to have the right scaling for Brownian motion). Our argument proceeds as follows. We first recall a bound due to Sperb.

\begin{thm}[Sperb's P-function \cite{sperb}]  Let $\Omega \subset \mathbb{R}^2$ be convex and assume that $u$ is the function satisfying $-\Delta u = 2$ with Dirichlet boundary conditions on $\partial \Omega$. Then
$$ P(u) = | \nabla u|^2 + 4u \qquad \mbox{assumes its maximum at a critical point of}~u.$$
Moreover, the gradient assumes its maximum on the boundary and thus
$$ \| \nabla u\|_{L^{\infty}(\Omega)} \leq 2 \| u\|^{1/2}_{L^{\infty}(\Omega)}.$$
\end{thm}
This shows that it suffices to get a good bound on the maximum of the function, this then automatically implies a gradient bound. 
An estimate due to Larry Payne now provides an upper bound on the maximum of the torsion function in terms of the gradient energy (see also Salakhudinov \cite{sal}, Sperb \cite{sperbpayne} and Payne \cite{pay0, payque}).
\begin{thm}[Payne \cite{paynegood}] Let $\Omega \subset \mathbb{R}^2$ be simply connected and assume that $u$ is the function satisfying $-\Delta u = 2$ with Dirichlet boundary conditions on $\partial \Omega$.
Then
$$ \|u\|_{L^{\infty}(\Omega)}^2 \leq \frac{1}{2\pi} \int_{\Omega}{ |\nabla u|^2 dx}.$$
\end{thm}

We recall that the torsion function on a disk of radius $1/\sqrt{\pi}$ in polar coordinates is given by $u(r) = 1/(2\pi) - r^2/2$ allowing us to compute
$$  \|u\|_{L^{\infty}(\Omega)}^2 = \frac{1}{4\pi^2} =  \frac{1}{2\pi}\int_{0}^{1/\sqrt{\pi}}{ r^2 \cdot 2r \pi dr}= \frac{1}{2\pi} \int_{\Omega}{ |\nabla u|^2 dx}$$
and thus Payne's inequality is sharp on the disk. Our next step is to invoke Integration by parts: since $-\Delta u = 2$, we have
$$ \|u\|_{L^{\infty}(\Omega)}^2 \leq \frac{1}{2\pi} \int_{\Omega}{ |\nabla u|^2 dx} = \frac{1}{\pi} \int_{\Omega}{u(x) dx}.$$
The third ingredient is a fairly recent stability statement for torsional rigidity. 
\begin{thm}[Brasco, De Philippis and Velichkov \cite{brasco}] Let $\Omega$ be an open set in $\mathbb{R}^2$,
let $\Omega^*$ be the symmetrized open set with the same volume and let 
$$ T(\Omega) = \int_{\Omega}{u(x) dx},$$
where $u$ solves $-\Delta u = 2$ in $\Omega$ with Dirichlet boundary conditions. Then, for some universal $\tau > 0$,
$$ T(\Omega^*) - T(\Omega) \geq \tau \mathcal{A}(\Omega)^3.$$
\end{thm}
Explicit estimates on the size of $\tau$ are available but presumably many orders of magnitude away from optimal \cite{brasco}
and we will not pursue an explicit quantitative estimate here. We also refer to the particularly nice survey \cite{brasco2}. In summary,
we now have
$$ \|u\|_{L^{\infty}(\Omega)}^2 \leq \frac{1}{\pi} \int_{\Omega}{u(x) dx} \leq \frac{1}{4\pi^2} - \frac{\tau}{\pi} \mathcal{A}(\Omega)^3.$$
Sperb's estimate implies
$$ \| \nabla u\|_{L^{\infty}(\Omega)} \leq 2 \| u\|^{1/2}_{L^{\infty}(\Omega)} = 2\left( \frac{1}{\sqrt{2\pi}} - \tau_2 \cdot \mathcal{A}(\Omega)^3\right).$$
We recall that all these arguments hold for the normalization $-\Delta u = 2$, the bound for $-\Delta u = 1$ is by a factor of 2 smaller.
\end{proof}

\subsection{Estimating Brownian Motion.} We now arrive at the heart of the argument. We recall that it suffices to show an improvement over
the gradient estimate $\|\nabla u\|_{L^{\infty}} \leq 1/\sqrt{2\pi}$ for domains that are close to a disk in terms of Fraenkel asymmetry.
Our goal is to estimate that largest derivative of 
\begin{align*}
-\Delta u &= 2 \qquad \mbox{inside}~\Omega\\
u &= 0 \qquad \mbox{on}~\partial \Omega
\end{align*}
on the boundary, where $\Omega \subset \mathbb{R}^2$ is a convex domain normalized to have measure $|\Omega|=1$ that satisfies $\mathcal{A}(\Omega) \leq \varepsilon_0$
for some $\varepsilon_0$ sufficiently small (but positive and universal). 
We proceed by estimating the
lifetime of Brownian motion itself, the beginning of our argument is adapted from \cite{jianfeng}. We estimate the
expected lifetime of Brownian motion started at distance $\varepsilon$ from the point $x_0 \in \partial \Omega$.  
Since $\Omega$ is convex, there is a supporting hyperplane (i.e. a line) in $x_0$ such that all of $\Omega$ is strictly contained
on one side of the hyperplane (and possibly the hyperplane itself). We pick a time $T$ (to be chosen later) and consider Brownian
motion running for $T$ units of time. Brownian motion in the plane can be decoupled into two independent (one-dimensional) Brownian motions 
in two coordinates.
 \begin{center}
\begin{figure}[h!]
\begin{tikzpicture}[scale=0.9]
\draw [thick] (0,0) -- (10,0);
\draw [very thick] (0,1)  to[out=330, in=180] (5,0) to[out=0, in=230] (9.5, 1.5);
\filldraw (5, 1) circle (0.04cm);
\draw [dashed] (5, 0.8) -- (5, 0);
\node at (4.8, 0.4) {$\varepsilon$};
\filldraw (5, 0) circle (0.04cm);
\node at (5.3, -0.2) {$x_0$};
\node at (9, -0.4) {supporting hyperplane};
\node at (6,2) {the convex set $\Omega$};
\end{tikzpicture}
\caption{An outline of the geometry.}
\end{figure}
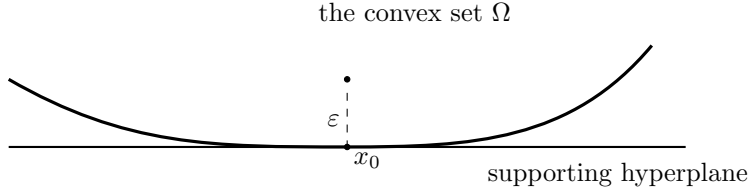
\end{center}

 We consider those as separate. For now we focus only, using the orientation suggested in the Figure, on the Brownian
motion acting in the $y-$coordinate. 
Clearly, if one starts rather close to the boundary, then the likelihood of hitting the boundary must be
rather large. This quantity is known and classical in probability theory.

\begin{lemma}[Reflection principle; e.g.~\cite{katz}] Let
  $\varepsilon > 0$ and let $T_{\varepsilon}$ be the hitting time for Brownian motion started at $B(0) = \varepsilon$,
  $T_{\varepsilon} = \inf\left\{ t > 0: B(t) = 0\right\}.$
  Then
$$ \mathbb{P}(T_{\varepsilon} \leq t) = 2 - 2\Phi\left( \frac{\varepsilon}{\sqrt{t}} \right),$$
where 
$$\Phi(z) = \frac{1}{\sqrt{2\pi}} \int_{-\infty}^{z} e^{-\frac{x^2}{2}} dx$$
 is the cumulative distribution function of $\mathcal{N}(0,1)$.
\end{lemma}
Exactly as in \cite{jianfeng}, we can use this identity to compute the density $\psi$ of the stopping time
$$ \psi(t) = \frac{d}{dt} \mathbb{P}(T_{\varepsilon} \leq t) =  \Phi' \left( \frac{\varepsilon}{\sqrt{t}} \right) \frac{\varepsilon}{t^{3/2}} = \frac{ \varepsilon}{\sqrt{2\pi} t^{3/2}} e^{-\frac{\varepsilon^2}{2t}}.$$
This distribution decays like $\sim t^{-3/2}$ and does not have a finite mean. However, the expected lifetime of particles hitting the threshold before time $T$ can be computed and is bounded by
$$ \mathbb{E}(T_{\varepsilon} \big| T_{\varepsilon} \leq T) = \int_{0}^{T}{ \psi(t) t dt} \leq \int_0^T { \frac{\varepsilon}{\sqrt{2\pi} t^{1/2}} dt} = \varepsilon \sqrt{\frac{2}{\pi}} \sqrt{T}.$$
We will work in the regime where $T > 0$ is fixed and $\varepsilon \rightarrow 0$. In that regime, we use
$$ \Phi \left( \frac{\varepsilon}{\sqrt{T}}\right) = \frac{1}{2} + \frac{1}{\sqrt{2\pi}}\frac{\varepsilon}{\sqrt{T}} + \mbox{l.o.t.}$$
and therefore
$$ \mathbb{P}\left(T_{\varepsilon} > T\right) = (1+o(1)) \sqrt{ \frac{2}{\pi}} \frac{\varepsilon}{\sqrt{T}} \qquad \mbox{as}~\varepsilon \rightarrow 0.$$
Here and henceforth, $o(1)$ will always refer to a quantity going to 0 as $\varepsilon \rightarrow 0$ for $T$ fixed.
We will now introduce a new ingredient into the mix. Instead of merely trying to control the number of particles that do not collide with the line up to time $T$,
we are also interested in the distribution of the survivors. For this purpose, we fix the interval $[0,\pi]$ and study the solution of the heat equation
$$ (\partial_t - \Delta) u = 0$$
with Dirichlet boundary conditions at 0 and $\pi$. The stochastic interpretation of the heat equation implies that the survival distribution at time $T$ is given
by the solution of $u(T,x)$ where $u(0,x) = \delta_{\varepsilon}$ is a Dirac mass in $\varepsilon$. We note that, in this formulation, we also introduce an
additional barrier at $\pi$ -- since we will only work in the setting of $\mathcal{A}(\Omega) \leq \varepsilon_0$, no part of $\Omega$ can cross the second
barrier and we obtain a valid upper bound.
The eigenfunctions of the Laplace operator $-\Delta$ with Dirichlet boundary conditions on $[0, \pi]$ are given by $\phi_k(x) = \sqrt{2/\pi}\sin{(k x)}$ and the eigenvalues are $k^2$. This implies that the heat kernel
can be written as (where we scale time a factor $1/2$ to scale correctly with respect to classical Brownian motion)
$$ p_t(x,y) = \frac{2}{\pi}\sum_{k=1}^{\infty}{e^{-k^2 t/2} \sin{(kx)} \sin{(ky)}}.$$
We are again working in a favorable regime: since we will keep time $T$ fixed and send $\varepsilon \rightarrow 0$, we obtain the survivor density
$$ p_{T}(\varepsilon, y) = \frac{2}{\pi}\sum_{k=1}^{\infty}{e^{-k^2 T/2} \sin{(k \varepsilon)} \sin{(ky)}} = (1+o(1))  \frac{2 \varepsilon}{\pi}\sum_{k=1}^{\infty}{e^{-k^2 T/2} k \sin{(ky)}}.$$
Since we have the explicit distribution, we should be able to recover the Lemma cited above and will now do so.
The total survival likelihood is given by
\begin{align*}
 \int_{0}^{\pi}{ p_T(\varepsilon, y) dy} &= (1+o(1))  \frac{2 \varepsilon}{\pi} \int_{0}^{\pi}{ \sum_{k=1}^{\infty}{e^{-k^2 T/2} k \sin{(ky)}}}\\
 &= (1+o(1))  \frac{4 \varepsilon}{\pi} \sum_{k \in \mathbb{N},~{\tiny \mbox{odd}}}{e^{-k^2 T/2} }
\end{align*}

\begin{center}
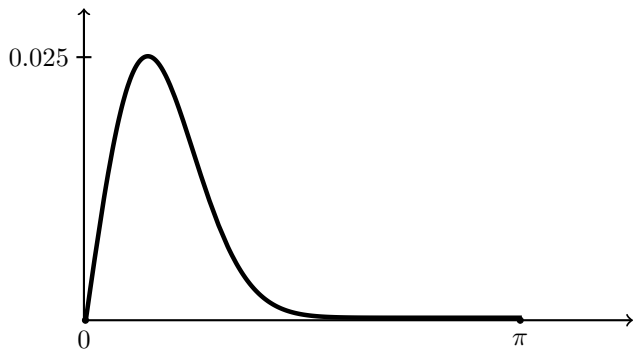
\begin{figure}[h!]
\begin{tikzpicture}
\node at (0,0,) {\includegraphics[scale=0.3]{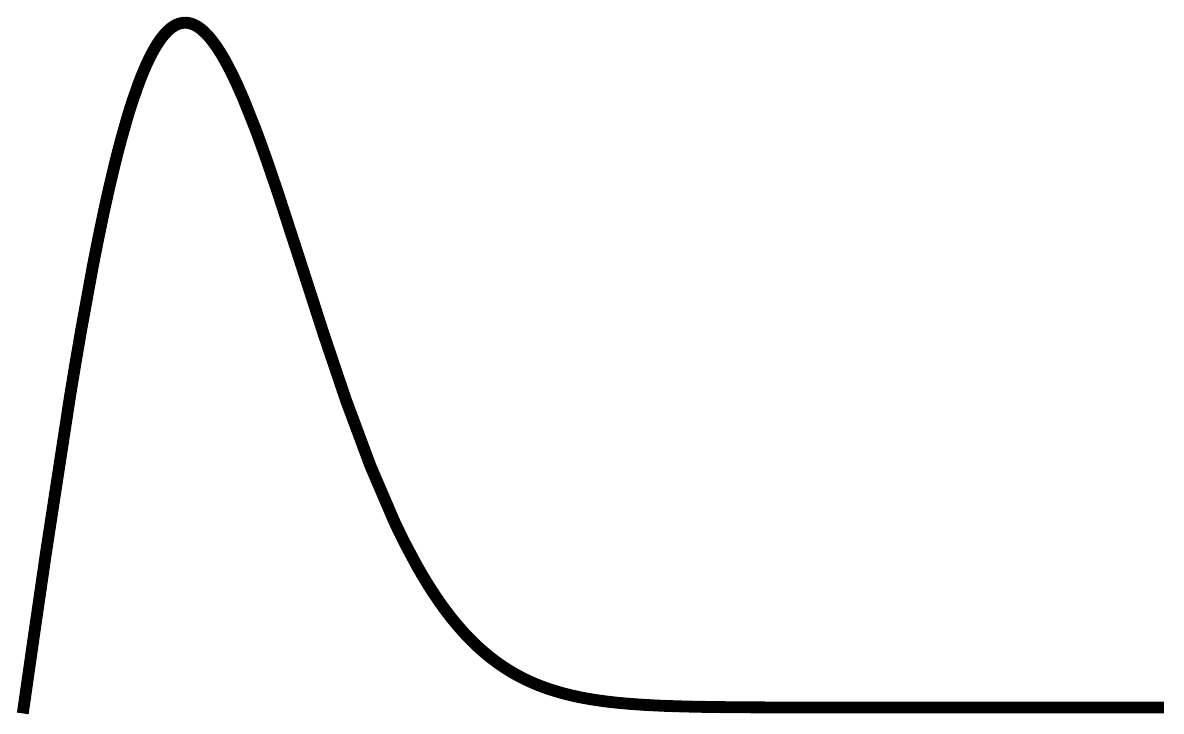}};
\draw [thick, ->] (-3.3,-2.15) -- (4,-2.15);
\draw [thick, ->] (-3.3,-2.15) -- (-3.3,2);
\filldraw (-3.28, -2.15) circle (0.04cm);
\node at (-3.3, -2.4) {0};
\filldraw (2.5, -2.15) circle (0.04cm);
\node at (2.5, -2.4) {$\pi$};
\draw [thick] (-3.3-0.1,1.35) -- (-3.3+0.1,1.35);
\node at (-3.9, 1.35) {$0.025$};
\end{tikzpicture}
\caption{The survivors when starting in $x=0.01$ and running until $T=0.2$. One observes that even though few survive, the survivors travel quite a distance. Moreover, the survival probability goes to 0 as $\varepsilon \rightarrow 0$ but the profile converges.}
\end{figure}
\end{center}

We would like this number to correspond to our survival estimate obtained above via the reflection principle. There is no reason it should since particles are
being absorbed on both ends of the interval. However, when $T$ is small, one would assume that the right-hand side of the interval has little to no effect. Indeed,
for $T \rightarrow 0$, we have
\begin{align*}
 (1+o(1))  \frac{4 \varepsilon}{\pi} \sum_{k~{\tiny \mbox{odd}}}{e^{-k^2 T/2} } &=  (1+o(1))  \frac{4 \varepsilon}{\pi} \sum_{k \in \mathbb{N},~{\tiny \mbox{odd}}}{e^{-(k \sqrt{T})^2/2} } \\
 &\sim \frac{2\varepsilon}{\pi} \int_0^{\infty}{e^{-(x\sqrt{T})^2/2} dx} \sim \sqrt{ \frac{2}{ \pi }} \frac{ \varepsilon}{\sqrt{T}}
\end{align*}
which recovers the estimate for the survival probability derived above for $T \rightarrow 0$. However, what is important for our successive argument is having direct access to
the distribution of the surviving particles.
Returning to the two-dimensional picture, we can think of running two-dimensional Brownian motion and then estimating the distribution of the survivors
after $T$ units of time. We do not have any idea where $\Omega$ is located, so we initially work in the larger domain created by the supporting hyperplane. We just derived the distribution along the
$y-$coordinate, the $x-$coordinate is an independent Brownian motion which has a Gaussian distribution after $T$ units of time.
This, in turn, implies that we have an estimate on the distribution of surviving Brownian motions in a half-space.

So far, we haven't made any use of the domain $\Omega$: clearly, it implies that our closed-form representation of survival probabilities is a strict upper
bound on the actual distribution within $\Omega$ (since some of the particles lie outside and some of the paths will leave $\Omega$ before returning).
Here, we use one more elementary Lemma (whose scaling could presumably be slightly improved but the simplest possible argument is enough for our
purposes here).
\begin{lemma}
 If $\Omega \subset \mathbb{R}^2$ is convex, has volume 1 and satisfies $\mathcal{A}(\Omega) \leq \delta_0$ where we assume $\delta_0 > 0$ to be sufficiently small but universal, then $\Omega \subset B$ for some ball $B$
satisfying 
$$ |B| \leq 1 + c \cdot \mathcal{A}(\Omega)^{1/2},$$
where $c$ is a universal constant depending only on $\delta$.
\end{lemma}
\begin{proof}
Let $B$ be the ball minimizing the Fraenkel asymmetry and let us assume for ease of exposition that $B$ is centered in 0. Since $|\Omega|=1$, we have $|B|=1$ and $B$
has radius $1/\sqrt{\pi}$. Suppose now that $x_0 \in \Omega$ and $\|x_0\| = 1/\sqrt{\pi} + z$ for $z$ sufficiently small. 

\begin{center}
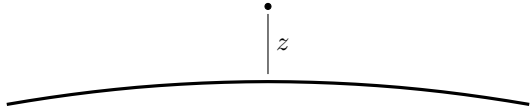
\begin{figure}[h!]
\begin{tikzpicture}
   \draw [black,very thick,domain=80:100] plot ({20*cos(\x)}, {20*sin(\x)});
   \filldraw (0, 21) circle (0.04cm);
   \draw (0, 20.1) -- (0, 20.9);
   \node at (0.2, 20.5) {$z$};
   \end{tikzpicture}
\caption{Proof of Lemma 4.}
\end{figure}
\end{center}
A rough estimate on the size of the convex hull induced by $z$ and the ball is that it is at least $z^2$. (A more refined estimate shows that one could do prove $\gtrsim z^{3/2}$ but this would not affect our main argument which is why we present the simpler estimate).
Therefore, $z \lesssim \mathcal{A}(\Omega)^{1/2}$. This shows the desired inequality.
\end{proof}

\subsection{The Main Argument.} Our goal is to show that if we start a Brownian motion within $\varepsilon$ distance from the boundary of a domain whose Fraenkel asymmetry is close to 0, then the expected lifetime can be bounded from above by $c \varepsilon$, where $c < 1/\sqrt{2\pi}$. There are essentially two different factors influencing the expected lifetime: in the beginning, when the particle is close to the boundary, there is a chance of impacting on the boundary close to the point where we start: this is an overwhelmingly local phenomenon, it does not depend very much on the global shape. After a while, the global shape becomes more important.

\begin{center}
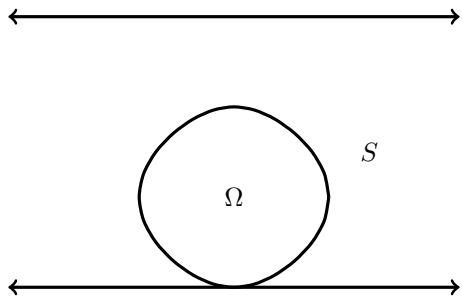
\begin{figure}[h!]
\begin{tikzpicture}[scale=0.6]
   \draw [very thick, <->] (-5, 4) -- (5, 4);
   \draw [black,very thick,domain=0:360, smooth] plot ({2*cos(\x)+0.1*cos(3*\x)}, {2*sin(\x)});
   \draw [very thick, <->] (-5, -2) -- (5, -2);
   \node at (0,0) {$\Omega$};
   \node at (3, 1) {$S$};
   \end{tikzpicture}
\caption{$\Omega$ is convex and $\mathcal{A}(\Omega)$ is close to 0 is contained in the strip of width $\pi$ (this sketch not to scale).}
\end{figure}
\end{center}

Our argument will mirror this fact: we fix a time $T$ (which we later set to be $T=0.13$). Up to time $T$, we do not care about the domain $\Omega$, we only care about the expected lifetime of particles within the strip (see Fig. 12) up to time $T$.  The strip is quite simple and we get explicit bounds for the number of particles that live up to time $T$, the expected hitting time of those who do not as well as \textit{the distribution of the survivors} (i.e. the density of Brownian motion that has not impacted on the boundary of the strip within the first $T$ units of time). From this, we bound the lifetime of Brownian motion as follows
\begin{enumerate}
\item for the ones that impact the boundary of the strip within $T$ units of time, we compute an upper bound
\item for the ones that survive up to time $T$ but are not in the domain, we pick $T$ as the upper bound of the lifetime
\item and for the ones that survive up to time $T$ and are within the domain at time $T$, we compute an upper bound on
their remaining lifetime.
\end{enumerate}
For the third part, we obviously do not know precisely what the domain looks like and will make use of Lemma 4.
We recall that the expected lifetime of Brownian motion in a disk of radius $1/\sqrt{\pi}$ centered at $(0, 1/\sqrt{\pi})$ is given by
$$ v(x,y) = \left(  \frac{1}{2\pi} - \frac{x^2 + (y - \pi^{-1/2})^2}{2} \right)_{+},$$
where we use the abbreviation $A_{+} = \max\left\{A, 0 \right\}$. The solution of the equation $-\Delta u = 2$ on a disk is continuous in the radius, it thus suffices
to show a sufficiently large gap for $\mathcal{A}(\Omega) = 0$. We will now collect all contributions: recall that
$$ \mathbb{E}(T_{\varepsilon} \big| T_{\varepsilon} \leq T) = \int_{0}^{T}{ \psi(t) t dt} \leq \int_0^T { \frac{\varepsilon}{\sqrt{2\pi} t^{1/2}} dt} = \varepsilon \sqrt{\frac{2}{\pi}} \sqrt{T}$$
and
$$ \mathbb{P}\left(T_{\varepsilon} > T\right) = (1+o(1)) \sqrt{ \frac{2}{\pi}} \frac{\varepsilon}{\sqrt{T}} \qquad \mbox{as}~\varepsilon \rightarrow 0.$$
We have
$$ \mathbb{E}~\mbox{lifetime} \leq  \mathbb{E}(T_{\varepsilon} \big| T_{\varepsilon} \leq T) \mathbb{P}(T_{\varepsilon} \leq T)  + \mathbb{E}(T_{\varepsilon} \big| T_{\varepsilon} > T) \mathbb{P}\left(T_{\varepsilon} >T\right).$$
We note that, for $T$ fixed and $\varepsilon \rightarrow 0$, $ \mathbb{P}(T_{\varepsilon} \leq T) $ converges to 1 at a rate linear in $\varepsilon$ and we can bound it from above by 1 without losing too much. This takes care of the first contribution: the expected lifetime of particles impacting within $T$ units of time on the boundary of the strip can be bounded from above by
$$ \mathbb{E}(T_{\varepsilon} \big| T_{\varepsilon} \leq T) \mathbb{P}(T_{\varepsilon} \leq T) \leq \mathbb{E}(T_{\varepsilon} \big| T_{\varepsilon} \leq T) \leq  \varepsilon \sqrt{\frac{2}{\pi}} \sqrt{T}.$$ 
The only quantity that remains to be analyzed is $ \mathbb{E}(T_{\varepsilon} \big| T_{\varepsilon} > T)$: this part we split into two contributions, the particles that are within the strip but outside of $\Omega$ and for which we bound the lifetime from above by $T$ (since they must have left the domain before that time) and the particles inside $\Omega$. For these particles, we use the Markov property: the torsion function $u$ within the domain tells us exactly for how much longer they are alive and thus
 $$ \mathbb{E}(T_{\varepsilon} \big| T_{\varepsilon} > T) \leq T +  \int_{\Omega}{s(x,y) u(x,y) dx dy}.$$

We can write down the conditional density function for the survivors within the two plates as $\varepsilon \rightarrow 0$ as
\begin{align*}
 s(x,y) &=  \frac{1}{\sqrt{2\pi T}} e^{-\frac{x^2}{2T}}   \sqrt{ \frac{2T}{\pi}}\sum_{k=1}^{\infty}{e^{-k^2 T/2} k \sin{(ky)}} \\
 &= \frac{ e^{-\frac{x^2}{2T}}}{\pi}   \sum_{k=1}^{\infty}{e^{-k^2 T/2} k \sin{(ky)}}.
 \end{align*}

 This is, finally, where we put the domain $\Omega$ into play. Denoting the solution of $-\Delta u = 2$ inside $\Omega$, we can use the Markov property of Brownian motion
 to write
This, a priori, seems tricky for two reasons: we do not know where exactly the domain $\Omega$ is and we do not know what the torsion function $u$ looks like.
The second problem is easy to address: by surrounding the domain $\Omega$ by a slightly larger ball (we refer to the Lemma above), we can simply bound the
torsion function from above by the torsion function in the larger ball. We do not know where the ball should be located and bypass this issue by simply taking the maximum over all translations of the ball. Put differently, we set
$$ u(x,y) = \frac{1}{2\pi} - \frac{(x-x_c)^2 + (y-y_c)^2}{2}$$
and write
 $$ \mathbb{E}(T_{\varepsilon} \big| T_{\varepsilon} > T) \leq T +  \max_{x_c, y_c} \int_{\Omega}{s(x,y) u(x,y) dx dy}.$$
If we can show that this quantity is strictly smaller than $1/\sqrt{2\pi}$ for some value $T$, then this is also true for slightly
larger radii by the continuity of all the involved objects. Picking $T=0.13$ leads to the upper bound $0.77$ and thus, divided
by 2, to the upper bound 0.385 that is strictly smaller than $1/\sqrt{2\pi}$.

\section{A Comment on the Numerics}
\subsection{The approach.} We conclude with a brief description of the numerical methods used to calculate the torsion function. Using classical potential theory, Poisson's equation 
\begin{align*}\label{eqn:torsion}
-\Delta u &= 1,\quad x \in \Omega, \\
u &= 0,\quad x \in \partial \Omega, \nonumber
\end{align*}
can be reduced to a boundary integral equation. In particular, writing 
$$u(x_1,x_2) = -\frac{x_2^2}{2} + \psi(x_1,x_2)$$
it follows that $\psi$ satisfies the boundary value problem
\begin{align*}
\Delta \psi(x) &= 0,\quad x \in \Omega, \\
\psi(x) &= \frac{x_2^2}{2},\quad x \in \partial \Omega,
\end{align*}
where $x= (x_1,x_2).$ 

The boundary value problem for $\psi$ can be reduced to a boundary integral equation using classical potential theory. In particular, we represent $\psi$ via the formula
\begin{align*}
\psi(x) = \int_{\partial \Omega} \frac{x-x'}{|x-x'|^2}\cdot n(x')\, \sigma(x')\,{\rm d}s_{x'},
\end{align*}
where $\sigma$ is an unknown function, which we will refer to as the density. We note that $\psi$ defined in this manner is harmonic in the interior of $\Omega$; all that is required is to choose $\sigma$ so that the boundary conditions are satisfied. Taking the limit as $x$ approaches a point $x_0$ on the boundary $\partial \Omega,$ we obtain
\begin{align*}
-\pi \sigma(x) + \int_{\partial \Omega} \frac{x-x'}{|x-x'|^2} \cdot n(x')\,{\rm d}s_{x'} = \frac{x_2^2}{2},
\end{align*}
for all $x_0 \in \partial \Omega.$ Existence and uniqueness of solutions to the above boundary integral equation for Lipschitz domains is a classical result. In this paper, when solving boundary integral equations of this form we use a method similar to that described in \cite{colton} when $\partial \Omega$ is twice-differentiable, and the approach described in \cite{hoskins} when $\Omega$ is a polygon.\\

\subsection{Details.} To obtain the lower bound stated above and shown in Figure 1 we use gradient descent to find the convex 128-sided polygon with unit area that maximizes $(\partial/\partial y) u (0,0)$ subject to the following constraints: the polygon is symmetric in $x,$ and each vertex $(x_i,y_i),$ $i=1,\dots,128$ is of the form
$$  (x_i, y_i) = r_i (\cos(\pi (i-1)/127),\sin (\pi(i-1)/127))$$
 for some positive number $r_i.$ In particular, the gradient is computed with respect to the parameters $r_1,\dots,r_{128}.$ The result is shown in Figure 1. In the optimization special care must be taken to preserve convexity. Here we did this by removing any vertex which was co-linear with its neighbors. Changing the number of vertices did not seem to increase the maximum substantially which might indicate that Figure 1 could be very close to the truth.
 
 \subsection{An alternative argument by Sweers.}  Guido Sweers (personal communication) has provided a different approach to the construction of extremal domains via conformal maps which we include with his permission. His approach leads to an explicitly given domain on which
 $$ \| \nabla u\|_{L^{\infty}(\Omega)} \sim 0.3561 | \Omega|^{1/2}.$$
 This is a little bit ($\sim 2 \cdot 10^{-3}$) smaller than our constant but has the advantage of coming with an explicit construction.
 Let $\Omega$ be a simply connected domain and let $h: E \rightarrow \Omega$ be a biconformal map. Then the solutions of
 $$ \begin{cases} - \Delta w = f \qquad &\mbox{in}~\Omega \\
 w=0 \qquad &\mbox{on}~\partial \Omega \end{cases}$$
 and 
 $$ \begin{cases} - \Delta u = \left| h'(\cdot)\right|^2 (f \circ h) \quad &\mbox{in}~E \\
 u=0 \qquad \qquad \qquad &\mbox{on}~\partial E \end{cases}$$
are related via
$$ (w \circ h)(x,y) = u(x,y).$$

\begin{figure}[h!]
\begin{minipage}[l]{.48\textwidth}
\begin{tikzpicture}
\node at (-2.5,0) {};
\node at (0,0) {\includegraphics[width = 0.5\textwidth]{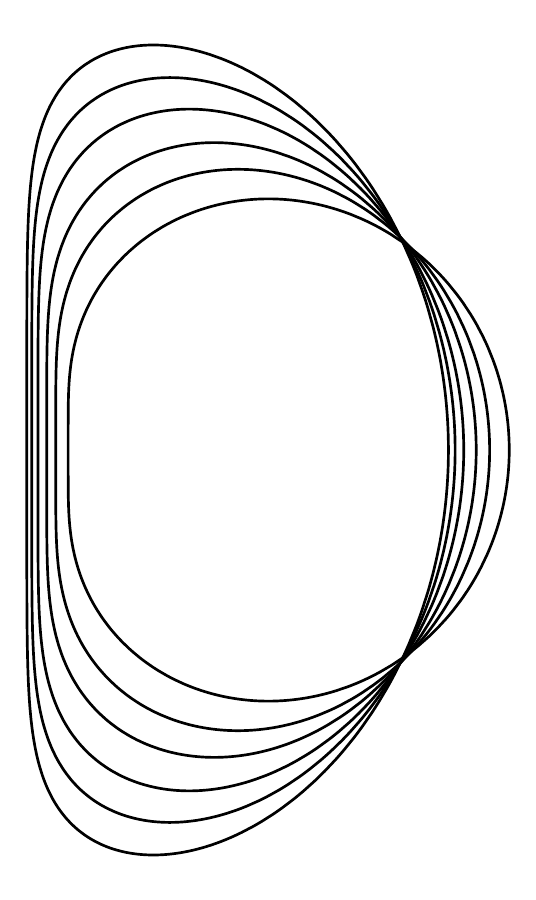}};
\end{tikzpicture}
\end{minipage} 
\begin{minipage}[r]{.48\textwidth}
\begin{tikzpicture}
\node at (-2.5,0) {};
\node at (0,0) {\includegraphics[width = 0.5\textwidth]{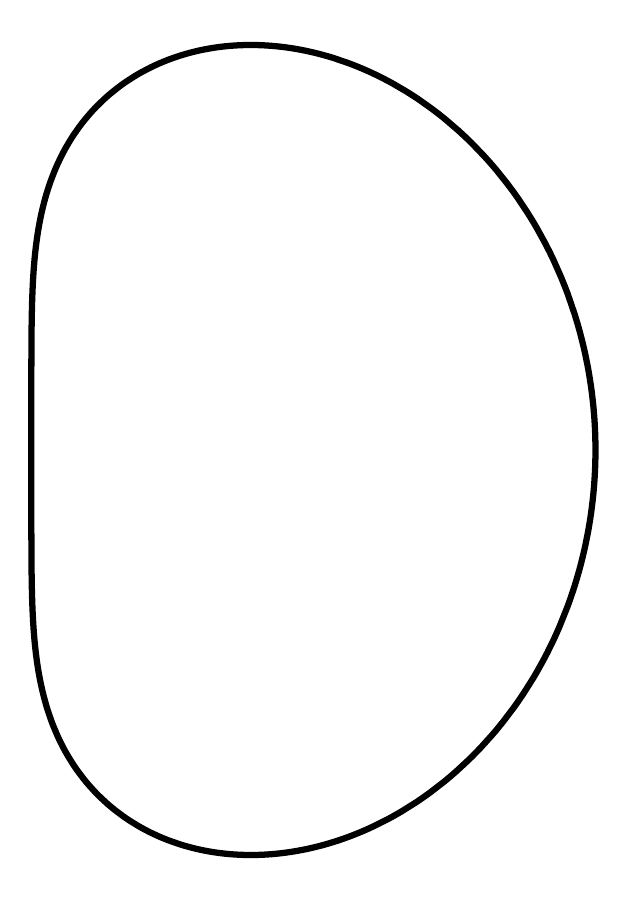}};
\end{tikzpicture}
\end{minipage} 
\caption{Left: $h_q(E_q)$ for $1 \leq q \leq 2$. Right: $h_q(E_q)$ for $q = 1.386$.} 
\end{figure}

This relationship is used to solve $-\Delta u = f$ on a domain $\Omega = h(E)$ where $E$ is an ellipse and $h$ is explicit. More precisely,
$$ E_q = \left\{ (x,y) \in \mathbb{R}^2: x^2 + (y/q)^2 \leq 1 \right\}$$
and 
$$ h_q(z) = z + \frac{3 q^2 + 1}{4 q^4 + 5 q^2 + 1} z^2 + \frac{1}{3(4q^2 + 1)} z^3.$$
$\Omega_q = h_q(E_q)$ is convex and $\partial \Omega_q$ is flat up to fourth order in $h_q(-1,0)$. The area formula gives
$$ \mbox{area}(h_q(E_q)) = \frac{\pi q (3q^2 + 1)(113 q^4 + 224 q^2 + 63)}{24 (q^2 + 1)(4 q^2 + 1)^2}.$$
Moreover, $| h_q'(x+ iy)^2|$ is a polynomial of order 4 in $(x,y)$ which allows one to solve
 $$ \begin{cases} - \Delta u_q = \left| h_q'(\cdot)\right|^2 \quad &\mbox{in}~E_q \\
 u=0 \qquad \qquad \qquad &\mbox{on}~\partial E_q \end{cases}$$
explicitly and the solution is a polynomial of order 6. One can write
$$ g(q) = \frac{1}{\sqrt{ \mbox{area}(h_q(E_q)) }} \frac{\partial u_q}{ \partial x} \big|_{h_q(-1,0)}.$$
One finds that $g$ is maximized for $q \sim 1.3866127$ resulting in $g(1.3866127) \sim 0.3561$. It is an interesting
question whether this approach can be refined by taking even higher order polynomials. Given how close this domain is to
our construction, it seems likely that the gain would be quite small.\\

\textbf{Acknowledgment.} We are grateful to Ren\'e Sperb for interesting discussions and grateful to Guido Sweers for allowing us to reproduce his construction.


\begin{thebibliography}{}

\bibitem{arango} J. Arango, A. Gomez,
Critical points of solutions to elliptic problems in planar domains. 
Commun. Pure Appl. Anal. 10 (2011), no. 1, 327--338.

\bibitem{baern} A. Baernstein, 
Symmetrization in analysis.
With David Drasin and Richard S. Laugesen. With a foreword by Walter Hayman. New Mathematical Monographs, 36. Cambridge University Press, Cambridge, 2019.

\bibitem{bandle} C. Bandle, 
On isoperimetric gradient bounds for Poisson problems and problems of torsional creep.
Z. Angew. Math. Phys. 30 (1979), no. 4, 713--715.

\bibitem{ban} R. Banuelos and T. Carroll, 
Brownian motion and the fundamental frequency of a drum.
Duke Math. J. 75 (1994), no. 3, 575--602.

\bibitem{ban2} R. Banuelos and B. Davis,
Heat kernel, eigenfunctions, and conditioned Brownian motion in planar domains.
J. Funct. Anal. 84 (1989), no. 1, 188--200.

\bibitem{ban3} R. Banuelos, M. van den Berg and T. Carroll, 
Torsional rigidity and expected lifetime of Brownian motion.
J. London Math. Soc. (2) 66 (2002), no. 2, 499--512.

\bibitem{ban4} R. Banuelos and T. Carroll, 
The maximal expected lifetime of Brownian motion.
Math. Proc. R. Ir. Acad. 111A (2011), no. 1, 1--11.

\bibitem{beck1} T. Beck, Uniform level set estimates for ground state eigenfunctions. SIAM J. Math. Anal. 50 (2018), no. 4, 4483--4502.

\bibitem{beck2} T. Beck, The torsion function of convex domains of high eccentricity, arXiv:1809.09212

\bibitem{beck3} T. Beck, Localization of the first eigenfunction of a convex domain, arXiv:1910.04905

\bibitem{hermitemany}  T. Beck, B. Brandolini, K. Burdzy, A. Henrot, J. Langford, S. Larson, R. Smits and S. Steinerberger, Improved Bounds for Hermite-Hadamard Inequalities in Higher Dimensions, Journal of Geometric Analysis, to appear.

\bibitem{bes}  M. Bessenyei: The Hermite-Hadamard inequality on simplices, American Mathematical
Monthly 115(4), 339--345 (2008)

\bibitem{brascamp} H. J. Brascamp and E. H. Lieb, On extensions of the Brunn-Minkowski and Prekopa Leindler theorems, including inequalities for log concave functions, and with an application to the diffusion
equation, J. Funct. Anal. 22 (1976), 366--389.

\bibitem{brasco} L. Brasco, G. De Philippis, B. Velichkov, Faber-Krahn inequalities in sharp quantitative form, Duke
Math. J. 164, no. 9, 1777--1831 (2015)

\bibitem{brasco2} L. Brasco, G. De Philippis,
Spectral inequalities in quantitative form. Shape optimization and spectral theory, 201--281, De Gruyter Open, Warsaw, 2017.

\bibitem{bonk} M. Bonk, 
On Bloch's constant.
Proc. Amer. Math. Soc. 110 (1990), no. 4, 889--894.

\bibitem{bourgain} J. Bourgain, On Pleijel's nodal domain theorem. Int. Math. Res. Not. 2015, no. 6, 1601--1612.

\bibitem{bous} J. Boussinesq, Etude nouvelle sur l'\'equilibre et le mouvement des corps solides \'elastiques dont certaines dimensions sont tr\`es petites par rapport a d'autres, J. Math. Pures Appl. 16 (1871), 125--274

\bibitem{cal1} J. de la Cal and J. Carcamo, 
Multidimensional Hermite-Hadamard inequalities and the convex order. 
J. Math. Anal. Appl. 324 (2006), no. 1, 248--261. 

\bibitem{cal2} J. de la Cal, J. Carcamo and L. Escauriaza, 
A general multidimensional Hermite-Hadamard type inequality.
J. Math. Anal. Appl. 356 (2009), no. 2, 659--663. 





\bibitem{chen} Y. Chen, Hadamard's inequality on a triangle and on a polygon. Tamkang J. Math. 35 (2004), no. 3, 247--254. 

\bibitem{colton}  D. Colton and R. Kress, Integral equation methods in scattering theory. Pure and Applied Mathematics (New York). A Wiley-Interscience Publication. John Wiley \& Sons, New York, 1983

\bibitem{square} S. S. Dragomir,
On the Hadamard's inequality for convex functions on the co-ordinates in a rectangle from the plane. 
Taiwanese J. Math. 5 (2001), no. 4, 775--788. 

\bibitem{drago} S. Dragomir and C. Pearce, Selected Topics on Hermite-Hadamard Inequalities and Applications, 2000.

\bibitem{dragomirk} S. Dragomir, G. Keady, 
A Hadamard-Jensen inequality and an application to the elastic torsion problem. 
Appl. Anal. 75 (2000), no. 3-4, 285--295. 


\bibitem{finn} D. Finn,)
Convexity of level curves for solutions to semilinear elliptic equations. 
Commun. Pure Appl. Anal. 7 (2008), no. 6, 1335--1343.


\bibitem{filon} ] L. N. G. Filon,  On the resistance to torsion of certain forms of shafting, with
special reference to the effect of keyways, Philos. Trans. Roy. Soc. London Ser. A 193 (1900), 309--352.

\bibitem{georgiev} B. Georgiev and M. Mukherjee, 
Nodal geometry, heat diffusion and Brownian motion. 
Anal. PDE 11 (2018), no. 1, 133--148.

\bibitem{grieser} D. Grieser and D. Jerison, Asymptotics of the first nodal line of a convex domain. Invent. Math. 125 (1996), no. 2, 197--219. 

\bibitem{grieser2} D. Grieser and D. Jerison, The size of the first eigenfunction of a convex planar domain. J. Amer. Math. Soc. 11 (1998), no. 1, 41--72.

\bibitem{griff} A. A. Griffith and G.I. Taylor: The use of soap films in solving torsion problems. In:
Batchelor, G.K. (ed.) The scientific papers of Sir Geoffrey Ingram Taylor, Vol. I, Mechanics of
Solids. Cambridge: Cambridge Univ. Press (1958). 

\bibitem{hada} J. Hadamard, \'Etude sur les propri\'et\'es des fonctions enti\`eres et en particulier d'une fonction consid\'er\'ee par Riemann, Journal de Math\'ematiques Pures et Appliqu\'ees 58, 1893, p. 171--215.

\bibitem{hamel} F. Hamel and E. Russ. Comparison results and improved quantified inequalities for semilinear
elliptic equations, Math. Ann. 367 (2017), 311--372.

\bibitem{henriot} A. Henriot, Extremum problems for eigenvalues of elliptic operators, Frontiers in Mathematics. Birkhauser Verlag, Basel, 2006.


\bibitem{hermite} C. Hermite, Sur deux limites d’une integrale define, Mathesis, 3 (1883), 82.

\bibitem{hoskins} J. G. Hoskins,  V. Rokhlin and K. Serkh, 
On the numerical solution of elliptic partial differential equations on polygonal domains.
SIAM J. Sci. Comput. 41 (2019), no. 4, A2552--A2578.

\bibitem{katz} I. Karatzas and S. Shreve, Brownian Motion and Stochastic Calculus, Springer Graduate Texts in Mathematics, 1991.

\bibitem{kawohl0} B. Kawohl, Rearrangements and convexity of level sets in PDE. Lecture Notes in Mathematics, 1150. Springer-Verlag, Berlin, 1985.

\bibitem{kawohl} B. Kawohl, On the location of maxima of the gradient for solutions to quasilinear elliptic problems and a problem raised by St Venant, J. Elasticity 17
(1987), 195--206.

\bibitem{kawohl2} B. Kawohl,
On the shape of solutions to some variational problems. Nonlinear analysis, function spaces and applications, Vol. 5 (Prague, 1994), 77--102, Prometheus, Prague, 1994.

\bibitem{keady} G. Keady and A. McNabb, 
The elastic torsion problem: solutions in convex domains. 
New Zealand J. Math. 22 (1993), no. 2, 43--64.

\bibitem{kim} D. Kim, Quantitative inequalities for the expected lifetime of Brownian motion, Michigan Math. J., to appear.

\bibitem{stv1} A. Kosmodemyanskii,
The behavior of solutions of the equation $\Delta u = -1$ in convex domains. (Russian)
Dokl. Akad. Nauk SSSR 304 (1989), no. 3, 546--548; translation in
Soviet Math. Dokl. 39 (1989), no. 1, 112--114

\bibitem{larson} S. Larson, A sharp multidimensional Hermite-Hadamard inequality, arXiv:2005.01853 

\bibitem{lieb0} E. Lieb,
On the lowest eigenvalue of the Laplacian for the intersection of two domains.
Invent. Math. 74 (1983), no. 3, 441--448.

\bibitem{lieb} E. Lieb and M. Loss, Analysis. Second edition. Graduate Studies in Mathematics, 14. American Mathematical Society, Providence, RI, 2001.

\bibitem{lierl} J. Lierl and S. Steinerberger,
A local Faber-Krahn inequality and applications to Schrodinger equations. 
Comm. Partial Differential Equations 43 (2018), no. 1, 66--81.

\bibitem{jerison} D. Jerison, Locating the first nodal line in the Neumann problem. Trans. Amer. Math. Soc. 352 (2000), no. 5, 2301--2317.


\bibitem{jianfeng} J. Lu and S. Steinerberger, A Dimension-Free Hermite-Hadamard Inequality via Gradient Estimates for the Torsion Function,  Proc. Amer. Math. Soc, to appear

\bibitem{jianfeng2} J. Lu and S. Steinerberger, Optimal Trapping of Brownian Motion: A Nonlinear Analogue of the Torsion Function, arXiv:1908.06273

\bibitem{makai} E. Makai, A proof of Saint Venant's theorem on torsional rigidity.
Acta Math. Acad. Sci. Hungar. 17 (1966), 419--422.

\bibitem{ma} X-N, Ma, S. Shi and Y. Ye, 
The convexity estimates for the solutions of two elliptic equations. 
Comm. Partial Differential Equations 37 (2012), no. 12, 2116--2137.

\bibitem{makar} L. G. Makar-Limanov, 
The solution of the Dirichlet problem for the equation $\Delta u = -1$ in a convex region. 
Mat. Zametki 9 (1971), 89--92.

\bibitem{nicu2} M. Mihailescu and C. Niculescu, 
An extension of the Hermite-Hadamard inequality through subharmonic functions. 
Glasg. Math. J. 49 (2007), no. 3, 509--514.

\bibitem{nicu} C. Niculescu, 
The Hermite-Hadamard inequality for convex functions of a vector variable.  
Math. Inequal. Appl. 5 (2002), no. 4, 619--623. 

\bibitem{choquet} C. Niculescu and L.-E. Persson, Old and New on the Hermite-Hadamard Inequality, Real Analysis Exchange 29, p. 663--686,  (2003-2004).


\bibitem{past} P. Pasteczka, Jensen-type Geometric Shapes, arXiv:1804.03688 

\bibitem{paynegood} L. E. Payne, Some isoperimetric inequalities in the torsion problem for multiply connected regions,  in:Studies in mathematical analysis and related topics, Stanford Univ. Press, 1962.

\bibitem{pay0} L. E. Payne,
Isoperimetric inequalities and their applications.
SIAM Rev. 9 (1967), 453--488.

\bibitem{pay1} L. E. Payne,
Bounds for the maximum stress in the Saint Venant torsion problem.
Special issue presented to Professor Bibhutibhusan Sen on the occasion of his seventieth birthday, Part I.
Indian J. Mech. Math. Special Issue Special Issue (1968/69), part I, 51--59.

\bibitem{payshape} L. E. Payne,
On two conjectures in the fixed membrane eigenvalue problem. 
Z. Angew. Math. Phys. 24 (1973), 721--729.

\bibitem{payque}  L.E. Payne, Some comments on the past fifty years of isoperimetric inequalities, in W.N. Everitt (ed.) Inequalities: fifty years on from Hardy,
Littlewood and Polya, Lecture Notes in Pure and Applied Mathematics 129,
Dekker, New York, 1991.

\bibitem{payphi} L. E. Payne and G. A. Philippin, 
Some applications of the maximum principle in the problem of torsional creep.
SIAM J. Appl. Math. 33 (1977), no. 3, 446--455.

\bibitem{payphi2} L. E. Payne and G. A. Philippin, 
Isoperimetric inequalities in the torsion and clamped membrane problems for convex plane domains.
SIAM J. Math. Anal. 14 (1983), no. 6, 1154--1162.

\bibitem{paywhe} L. E. Payne and L. T. Wheeler,
On the cross section of minimum stress concentration in the Saint Venant theory of torsion.
J. Elasticity 14 (1984), no. 1, 15--18.

\bibitem{phi} G. A. Philippin and V. Proytcheva, 
A minimum principle for the problem of St-Venant in $\mathbb{R}^N, N \geq 2$, 
Z. Angew. Math. Phys. 63 (2012), no. 6, 1085--1090.

\bibitem{pol1} G. Polya, Liegt die Stelle der grossten Beanspruchung an der Oberflache?,
Z. Angew. Math. Mech. 10 (1930), 353--360.

\bibitem{polya0} G. Polya, Torsional rigidity, principal frequency, electrostatic capacity and symmetrization. Quart. Appl. Math. 6, (1948). 267--277.

\bibitem{polya} G. Polya and G. Szego, Isoperimetric Inequalities in Mathematical Physics. Annals of Mathematics Studies, no. 27, Princeton University Press, Princeton, N. J., 1951. 

\bibitem{polind} G. Polya, 
Two more inequalities between physical and geometrical quantities.
J. Indian Math. Soc. (N.S.) 24 (1960), 413--419 (1961).

\bibitem{manas} M. Rachh and S. Steinerberger, On the location of maxima of solutions of Schrodinger's equation.
Comm. Pure Appl. Math. 71 (2018), no. 6, 1109--1122.

\bibitem{stv2} M. Ramaswamy, 
On a counterexample to a conjecture of Saint Venant.
Nonlinear Anal. 15 (1990), no. 9, 891--894.

\bibitem{st}  A. J.-C. Barre de Saint Venant, M\'{e}moire sur la torsion des prismes, M\'{e}moires pr\'{e}sent\'{e}s par divers savants \`{a} l'Acad\'{e}mie des Sciences, 14 (1856), pp. 233--560.

\bibitem{sal} R. Salakhudinov, 
Payne type inequalities for Lp-norms of the warping functions. 
J. Math. Anal. Appl. 410 (2014), no. 2, 659--669.

\bibitem{sperb} R. Sperb,Maximum principles and their applications,  Mathematics in
Science and Engineering, vol. 157, Academic Press, New York, 1981.

\bibitem{sperbpayne} R. Sperb,
Extension of two inequalities of Payne,
J. Inequal. Appl. 1 (1997), no. 2, 165--170.

\bibitem{sperb2} R. Sperb, 
Bounds for the first eigenvalue of the elastically supported membrane on convex domains. 
Special issue dedicated to Lawrence E. Payne.
Zeitschrift f\''ur Angew. Math. Phys. 54 (2003), no. 5, 879--903.

\bibitem{stv3} G. Sweers,
A counterexample with convex domain to a conjecture of de Saint Venant.
J. Elasticity 22 (1989), no. 1, 57--61.

\bibitem{stv4} G. Sweers, On examples to a conjecture of De Saint Venant, Nonlinear Analysis 18, p. 889--891, 1992.

\bibitem{stein1} S. Steinerberger, The Hermite-Hadamard inequality in higher dimension, Journal of Geometric Analysis, to appear.

\bibitem{talenti}  G. Talenti, Elliptic equations and rearrangements. Ann. Scuola Norm. Sup. Pisa Cl. Sci. (4) 3 (1976), no. 4, 697--718. 

\bibitem{tt}  W. Thomson and P.G. Tait, Treatise on natural philosophy, 1st edn., Oxford, 1867
\end{thebibliography}
\end{document}